\documentclass[10pt]{amsart}

\usepackage{amssymb,latexsym,amsmath,amsfonts}
\usepackage[mathscr]{eucal}
\usepackage{mathrsfs}
\usepackage{pxfonts}
\usepackage{amscd}
\usepackage{amsthm}
\usepackage{amsxtra}
\usepackage[nointegrals]{wasysym}
\usepackage{bm}
\usepackage{calc}
\usepackage{float}
\usepackage[usenames,dvipsnames]{xcolor}
\usepackage{hyperref}
\usepackage{hyperref}
\hypersetup{colorlinks=true, citecolor=MidnightBlue, linkcolor=BrickRed}

\numberwithin{equation}{section}
\theoremstyle{definition}
\newtheorem{definition}{Definition}[section]
\theoremstyle{definition}

\newtheorem{remark}[definition]{Remark}

\theoremstyle{plain}
\newtheorem{theorem}[definition]{Theorem}
\newtheorem{lemma}[definition]{Lemma}

\newtheorem{Prop}[definition]{Proposition}

\newcommand{\beas}{\begin{eqnarray*}}
\newcommand{\eeas}{\end{eqnarray*}}
\newcommand{\bes} {\begin{equation*}}
\newcommand{\ees} {\end{equation*}}
\newcommand{\be} {\begin{equation}}
\newcommand{\ee} {\end{equation}}
\newcommand{\bea} {\begin{eqnarray}}
\newcommand{\eea} {\end{eqnarray}}

\newcommand{\zt}{\zeta}

\newcommand{\om}{\omega}
\newcommand{\bphi}{\boldsymbol{\varphi}}

\newcommand{\dbar}{\overline\partial}

\newcommand{\Om}{\Omega}
\newcommand{\D}{\mathbb{D}}

\newcommand{\rea}{\operatorname{Re}}
 
\newcommand{\supp}{\operatorname{supp}}
\newcommand{\Ev}{\operatorname{Ev}}

\newcommand{\A}{\mathcal{A}}
\newcommand{\E}{\mathcal{E}}

\newcommand{\Ha}{\mathbb{H}}

\newcommand{\h}{\mathfrak{H}}
\newcommand{\bk}{\mathbf k}

\newcommand{\T}{\mathbb{T}}

\newcommand{\zbar}{\overline z}
\newcommand{\wbar}{\overline w}

\newcommand{\wt}{\widetilde}

\newcommand{\cont}{\mathcal{C}}
\newcommand{\hol}{\mathcal{O}}
\newcommand{\cH}{\mathcal{H}}

\newcommand{\CC}{\mathbb{C}^2}
\newcommand{\Cn}{\mathbb{C}^n}

\newcommand{\C} {\mathbb{C}} 
\newcommand{\Z} {\mathbb{Z}}

\newcommand{\N} {\mathbb{N}}

\begin{document}

\title{Hardy Spaces for a class of singular domains}
\author{A.-K. Gallagher}
\address{Gallagher Tool \& Instrument LLC, Redmond, WA}
\email{anne.g@gallagherti.com}
\author{P. Gupta}
\address{Department of Mathematics, Indian Institute of Science, Bangalore}
\email{purvigupta@iisc.ac.in}
\author{L. Lanzani}
\address{Department of Mathematics, Syracuse University, Syracuse, NY}
\email{llanzani@syr.edu}
\author{L. Vivas}
\address{Department of Mathematics, The Ohio State University, Columbus, OH}
\email{vivas@math.osu.edu}
\thanks{L. Lanzani and L. Vivas  were supported in part by the National Science Foundation (DMS-1901978 and DMS-1800777). P. Gupta was supported in part by a UGC CAS-II grant (Grant No. F.510/25/CAS-II/
2018(SAP-I)). Part of this work took place at the Banff International Research station during a workshop of the {\em Women in Analysis} (WoAn), an AWM Research Network.
We are grateful to the Institute for its kind hospitality and to the Association of Women in Mathematics for its generous support. We also wish to thank Mei-Chi Shaw for providing the inspiration for this work, Bj{\"o}rn Gustafsson for offering helpful feedback on an earlier version of this manuscript, and the anonymous referee for their useful comments.}

\begin{abstract}
We set a framework for the study of Hardy spaces inherited by complements of analytic hypersurfaces in domains with a prior
Hardy space structure. The inherited structure is a filtration, various aspects of which are studied in specific settings. For punctured planar domains, we prove a generalization of a famous rigidity lemma of 
Kerzman and Stein. A stabilization phenomenon is observed for egg domains. Finally, using proper holomorphic maps, we derive a filtration of Hardy spaces for certain power-generalized Hartogs 
triangles, although these domains fall outside the scope of the original framework.
\end{abstract}

\maketitle

\section{Introduction}\label{S:Intro}

In this paper, we construct Hardy spaces for a class of domains, which includes the punctured unit disk $\D^*=\D\setminus\{0\}$ and the product domain $\D\times\D^*$ as particularly simple, but 
enlightening, examples. Although our class of domains is not biholomorphically invariant, it is possible to push the construction forward under certain biholomorphisms. This allows us to construct Hardy spaces 
for the Hartogs triangle, $\Ha=\{(z_1,z_2)\in\C^2:|z_1|<|z_2|<1\}$, and compute the relevant Szeg{\H o} kernels. In fact, this was the original motivation for this work. The Hartogs triangle is of classical 
importance in several complex variables, see \cite{Sh15}, and serves as an important example of a singular domain since its boundary fails to be even locally graph-like at one point. While $\Ha$ and its 
generalizations have received a lot of attention from the point of view of the $\overline\partial$-problem, e.g., \cite{ChCh91, MaMi91,ChSh13,LaSh19,ChMc20}, and Bergman spaces, e.g., 
\cite{ChZe16,Ed16,EdMc16,EdMc17,ChEdMc19,HuWi20,NaPr20}, Hardy spaces for $\Ha$ were considered for the first time only recently by Monguzzi in \cite{Mo19}. Independently of Monguzzi, we had 
constructed a different Hardy space for the Hartogs triangle, and this discrepancy led us to recognize the central phenomenon of this paper. Before we describe this phenomenon, we clarify the main terminology 
used in this work. 

Since there is no unified notion of a Hardy space in the literature, we state here our minimum criteria for using this term. A Hilbert space of functions $\mathfrak H$ on the boundary of a domain is deemed a 
Hardy space {\em only if} there is a reproducing kernel Hilbert space (in the sense of Aronszajn in \cite{A50}) $\mathfrak X$ of holomorphic functions on the domain such that 
\begin{itemize}
	\item [$(a)$] functions in a dense subspace $\A\subset\mathfrak X$ admit boundary values in $\mathfrak H$, and
	\item [$(b)$] this identification of $\A$ with a subspace of $\mathfrak H$ is an isometry that extends to an isometric isomorphism between $\mathfrak X$ and $\mathfrak H$.
	\end{itemize}
We note that in all the explicit examples in this paper, $\mathfrak X$ is directly defined in terms of an exhaustion procedure on the domain, see Sections \ref{sec_examples}, \ref{S:Stabilization}, and 
\ref{S:Hartogs}. However, our general setting is not conducive to this process, and $\mathfrak X$ is only abstractly defined, for more details, see \eqref{E:corresp} and subsequent paragraphs.

To describe the class of domains under consideration, we start with a domain $\Om\Subset\Cn$ and a Borel measure $\nu$ supported on its boundary, $b\Om$, that admits a Hardy space structure. This 
structure is then inherited by domains that are obtained from $\Omega$  by  removing analytic hypersurfaces that are component-wise minimally defined, see Definition~\ref{D:vardeldom}. We refer to any such 
domain as a `hypersurface-deleted domain', and denote it by $\Om^*$. We call this process the `inheritance scheme', and the pair 
$(\Om, \nu)$ the `parent space'. 

In a notable departure from the classical theory, it turns out that under appropriate assumptions on the parent space, 
any hypersurface-deleted domain is associated to a filtration
of Hardy spaces, as opposed to a single such space. This is due to the fact that functions
 holomorphic on $\Om^*$ can be singular along the deleted hypersurface, but all orders of singularities cannot be captured in a single Hardy space, see the discussion at the beginning of Subsection~\ref{sec_D*}.
We demonstrate via explicit examples that this filtration may or may not stabilize, depending on the choice of $\nu$ and the deleted hypersurface.  

\subsection{Function-theoretic context} In \cite{Po13}, Poletsky and Stessin give a construction of Hardy spaces for hyperconvex domains in $\Cn$. We note that, while $\Om^*$ is pseudoconvex whenever $
\Om$ is, it is never hyperconvex. Our construction therefore covers a new class of domains. 

Note that this class of domains is however uninteresting from the point of view of Bergman space theory, since the Bergman space for $\Om^*$ equals the Bergman space for $\Om$, see \cite[Proposition 1.14]
{Oh02}. Additionally, our approach does not lead to meaningful Hardy spaces of {\em harmonic} functions because $b\Om$ is not, in general, a uniqueness set for harmonic functions on $\Om^*$. For instance, 
if $\Om^*=\D\setminus\{0\}$, then $\rea z$ and $\rea \frac{1}{z}$ are both harmonic on $\D^*$ but coincide on $b\D$. 

\subsection{Boundary-based approach to Hardy spaces}
 The lack of a general exhaustion procedure to construct $\mathfrak X$ shifts the burden of the construction to the dense subspace $\A$. In the classical 
setting of the unit disk, $\A$ is the disk algebra, i.e., the space of holomorphic functions on $\D$ that are continuous up to the boundary. If we extend this definition verbatim to the punctured disk, since $
\overline{\D^*}=\overline\D$, it would lead to the same Hardy space, which does not capture a significant class of holomorphic functions on $\D^*$. Our construction overcomes this issue. When $(\Om,\nu)$ 
is the parent space, we consider $\A$ to be
	\bes
		\A(\Om,\nu):=\hol(\Om)\cap\cont(\Om\cup\operatorname{supp}\nu).
	\ees 
Moreover, for $(\Om^*,\nu)$, we work with {\em subspaces} of 
				$\hol(\Om^*)\cap\mathcal{C}(\Om^*\cup\operatorname{supp}\nu )$
	which have prescribed singularity along the deleted hypersurface. Under appropriate assumptions (see Definition~\ref{D:Hardyspace2} for details), the $L^2(\nu)$-completion of $\A|_{\operatorname{supp}
	\nu}$ is a reproducing kernel Hilbert space on the domain in consideration. Hence, we call it a Hardy space and refer to its reproducing kernels as a Szeg{\H o} kernels.

	 We point out that there may be kernel functions $c(z, \cdot)$ that have the {\em reproducing property for $\A$}, namely, for all $z$ in the domain
	\be\label{E:repprop}
		F(z)=\int_{\operatorname{supp}\nu}F(w)\cdot c(z,w)\: d\nu(w)\quad \forall F\in\A, 
	\ee
but are not the Szeg\H o kernel for the associated Hardy space. For instance, this is the case for the Cauchy kernel of any smoothly bounded planar domain $\Om\neq \mathbb D$. Our boundary-based approach is particularly 
suited to obtaining such boundary integral representation formulas.
		  
\subsection{Description of results}  We first state conditions on the parent space $(\Om, \nu)$ that lead to a Hardy space for $\Om$, see Definitions~\ref{D:weaklyadmissible} and ~\ref{D:stronglyadmissible}. 
Then we provide the inheritance scheme that gives a filtration of Hardy spaces for  $\Om^*$, see Theorem~\ref{T:admissibleheir}. For each level of the filtration, we produce new kernels that have the 
reproducing property \eqref{E:repprop}. Moreover, we give a sufficient condition for these kernels to agree with the Szeg\H o kernels, see Proposition \ref{P:kernel}. We then proceed to analyse the framework 
via some examples.  

In Theorem \ref{T:rigidity}, we consider simply connected planar domains with finitely many points removed. For this class of domains, we formulate and extend a famous rigidity lemma of Kerzman and Stein 
\cite{KeSt78}, i.e., if $\Om\Subset \mathbb C$ is simply connected then the Cauchy kernel on $\Om$ coincides with the Szeg\H o kernel for $\Om$ if and only if $\Om$ is a disc.   
We next identify a family of domains for which the filtration of Hardy spaces stabilizes. These are egg domains, sometimes known as complex ellipsoids, in $\mathbb C^2$ from which a single hyperplane has been deleted, and we observe that the stabilization occurs at different levels depending on the choice of boundary measure, see Theorem \ref{thm_eggs}. Finally, we use proper holomorphic maps to transfer the filtered Hardy space structure on 
$\D\times \D^*$ to a class of non hypersurface-deleted domains, i.e., the Hartogs triangle and its rational power generalizations that were first introduced in \cite{Ed16, EdMc16}. We also produce explicitly the Szeg\H o kernels for these domains in Theorems 
\ref{T:Hartogs} and \ref{T:pgHartogs}.

\subsection{Structure of this paper.} In Section \ref{sec_examples}, we consider the punctured disk as this exemplifies the general construction of the filtration of Hardy spaces.
In Section \ref{S:varietydeleteddomains}, we provide the general framework and 
prove the main inheritance results.  Section \ref{S:planar} is specialized to the setting of planar domains, for which more explicit formulas can be proved by means of conformal mapping, along with the 
aforementioned rigidity result. The egg domains are dealt with in Section \ref{S:Stabilization}, and $\D\times\D^*$, the Hartogs triangle and its  rational power generalizations are treated in Section \ref{S:Hartogs}.

\section{Motivating example}\label{sec_examples}

We consider the open unit disk $\mathbb{D}$ and the arc-length measure $\sigma_{S^1}$ on $b\D$ as the parent space, and the punctured disk $\mathbb{D}^*:=\mathbb{D}\setminus\{0\}$ as the 
hypersurface-deleted domain. Using the basic descriptions for the $L^2$-Hardy space for the disk detailed in Subsection \ref{sec_D}, we derive a filtration of Hardy spaces for $(\mathbb{D}^*,\sigma_{S^1})$ in 
Subsection \ref{sec_D*}.  Throughout this section, we omit $\sigma_{S^1}$ from the notation for the relevant function spaces.

\subsection{Hardy Space for the unit disk}\label{sec_D}
The classical $L^2$-Hardy space $\mathcal{H}^2(\mathbb{D})$ is the space of holomorphic functions on $\D$ that are finite in the norm given by
  $$
		\|F\|_{\mathcal{H}^2(\mathbb{D})}
  :=\sup_{0<r<1}\left(\frac{1}{2\pi}\int_{0}^{2\pi} |F(re^{i\theta})|^2\;d\theta\right)^{\frac{1}{2}}.
 $$
Note that for any $F\in\mathcal{H}^2(\mathbb{D})$ with $F(z)=\sum_{j=0}^\infty a_j z^j$, it follows that 
$$\|F||_{\mathcal{H}^2(\mathbb{D})}=\Bigl(\sum_{j=0}^\infty|a_j|^2 \Bigr)^{\frac{1}{2}}<\infty.$$
 This characterization facilitates the identification of $\mathcal{H}^2(\mathbb{D})$ as a reproducing kernel Hilbert space, by way of considering the inner product 
	\bes
		\left<F,G\right>:=\lim_{r\rightarrow 1}\frac{1}{2\pi}\int_{0}^{2\pi} F(re^{i\theta})\,\overline {G(re^{i\theta})}\;d\theta
			\quad \quad \text{for}\ F,G\in \mathcal{H}^2(\mathbb{D}),
	\ees
and the evaluation operators $F\mapsto F(z)$ for  $z\in\mathbb{D}$
 and $F\in\mathcal{H}^2(\mathbb{D})$. Moreover, a truncation of power series argument gives that the disk algebra
$\mathcal{A}(\mathbb{D})=\mathcal{O}(\mathbb{D})\cap\mathcal{C}(\overline{\mathbb{D}})$ is  a dense subspace of $\mathcal{H}^2(\mathbb{D})$.
Next, the restriction to the boundary map from $\mathcal{A}(\mathbb{D})\subset\mathcal{H}^2(\mathbb{D})$ to $\mathcal{A}(\mathbb{D})|_{b\mathbb{D}}\subset L^2(b\mathbb{D})$ extends to an isometric isomorphism, up to a multiplicative constant, 
\begin{align*}
 \Phi: \mathcal{H}^2(\mathbb{D}) &\longrightarrow \overline{\mathcal{A}(\mathbb{D})|_{b\mathbb{D}}}^{L^2(b\mathbb{D})} \\
 F(z)=\sum_{j=0}^\infty a_j z^j&\mapsto \Phi(F)(e^{i\theta})=\sum_{j=0}^\infty a_j e^{ij\theta},
 \end{align*}
 where $\sum_{j=0}^\infty a_j e^{ij\theta}$ is the representation of $\Phi(F)$ as its Fourier series.
  We call  the closure of $\mathcal{A}(\mathbb{D})|_{b\mathbb{D}}$ in $L^2(b\mathbb{D})$ the Hardy space $\h^2(\mathbb{D})$ for $(\mathbb{D},\sigma_{S^1})$. Note that if we set $\mathfrak X$ as $(\cH^2(\D),\frac{1}{\sqrt{2\pi}}||.||_{\cH^2(\D)})$ and $\A$ as $\A(\D)$, then $\h=\h^2(\D)$ satisfies the minimum criterion of a Hardy space stated in the introduction.  

The Szeg{\H o} kernel $s$ for $\h^2(\mathbb{D})$ may now be derived from the Cauchy integral formula for $F\in\A(\D)$, which says that
$$F(z)
 =\frac{1}{2\pi i}\int_{b\mathbb{D}}\frac{F(w)}{w-z}\,dw
=\frac{1}{2\pi}\int_{b\mathbb{D}}\frac{F(w)}{1-z\overline{w}}\,d\sigma_{S^1}(w).$$ 
Since $s $ is uniquely determined by such a reproducing property and the fact that $\overline{s(z,.)}\in\h^2(\mathbb{D})$ for $z\in\mathbb{D}$, see Proposition~\ref{P:Riesz-bd}, we have that  
 $$ s(z,w)=\frac{1}{2\pi}\frac{1}{1-z\overline{w}}\;\;\text{ for } z\in\mathbb{D}, w\in b\mathbb{D}.$$

 \subsection{Hardy spaces on the punctured disk}\label{sec_D*}
 
 In an attempt to develop a Hardy space theory for the punctured disk, one might first consider $\mathcal{O}(\mathbb{D}^*)\cap\mathcal{C}(\overline{\mathbb{D}^*})$. However, 
 $\overline{\mathbb{D}^*}=\overline{\mathbb{D}}$, so this approach would only lead to the rediscovery of the Hardy space on the unit disk.  One might also try
 to construct a Hardy space for $\mathbb{D}^*$ by considering the closure of  $\bigl(\mathcal{O}(\mathbb{D}^*)\cap\mathcal{C}(\mathbb{D}^*\cup b\mathbb{D})\bigr)_{|_{b\mathbb{D}}}$ with respect to
 $L^2(b\mathbb{D})$. This fails, too, as pointwise evaluation on this class of $L^2(b\mathbb{D})$-functions is not bounded for any point in $\D^*$. To wit, consider the functions
$$F_k(z):=\sum_{j=1}^k\frac{1}{jz^{j}},\;\;k\in\mathbb{N}.$$
Clearly, $F_k\in\mathcal{O}(\mathbb{D}^*)\cap\mathcal{C}(\mathbb{D}^*\cup b\mathbb{D})$ , while
$$\|(F_k)_{|_{b\mathbb{D}}}\|_{L^2(b\mathbb{D})}\leq\sqrt{2\pi}\Bigl(\sum_{j=1}^\infty j^{-2}\Bigr)^{\frac{1}{2}}<\infty\;\;\forall k\in\mathbb{N}.$$ 
Since $F_k(z)$ diverges as $k\to\infty$ for any $z\in\mathbb{D}^*$, it follows that the pointwise evaluation operator is not a bounded operator on 
$\bigl(\bigl(\mathcal{O}(\mathbb{D}^*)\cap\mathcal{C}(\mathbb{D}^*\cup b\mathbb{D})\bigr)_{|_{b\mathbb{D}}},\|.\|_{L^2(b\mathbb{D})}\bigr) $ for any point in $\mathbb{D}^*$.
This failure stems from allowing holomorphic functions on $\mathbb{D}^*$ with essential singularities at the origin.
Thus, we allow poles of prescribed order at the origin, that is, for $k\in\mathbb{N}_{0}$, consider the following subset of 
$\mathcal{O}(\mathbb{D}^*)\cap\mathcal{C}(\mathbb{D}^*\cup b\mathbb{D})$
\begin{align}\label{E:AkD*}
  \mathcal{A}_k(\mathbb{D}^*)=\left\{F:\mathbb{D}^*\cup b\mathbb{D}\longrightarrow\mathbb{C}:
  F(z)=\left(z^{-k}G(z)\right)|_{\mathbb{D}^*} \text{ for some } G
  \in\mathcal{A}(\mathbb{D}) \right\}.
\end{align}
For each $k\in\mathbb{N}_{0}$ define $\h^2_k(\mathbb{D}^*)$ to be the closure of $\mathcal{A}_k(\mathbb{D}^*)|_{b\D}$ with respect to $L^2(b\mathbb{D})$. It immediately follows from $z|_{b\mathbb{D}}\neq 0$ that
\begin{align*}
  \h_k^2(\mathbb{D}^*)=\left\{f\in L^2(b\mathbb{D}): f=z^{-k}g\text{ for some }g\in\h^2(\mathbb{D}) \right\}.
\end{align*}
In particular, any function $f\in\h_k^2(\mathbb{D}^*)$ is represented by its Fourier series  $\sum_{j=-k}^\infty \hat{f}_j e^{ij\theta}$ where $\sum_{j=-k}^\infty|\hat{f}_j|^2<\infty$. 
Note that $\h^2_0(\mathbb{D}^*)=\h^2(\mathbb{D})$, $\h_k^2(\mathbb{D}^*)\subsetneq\h_{k+1}^2(\mathbb{D}^*)$ for any $k\in\mathbb{N}_0$, and
$\bigcup_{k=0}^\infty\h^2_k(\mathbb{D}^*)$ is dense in $L^2(b\mathbb{D})$.

We can also derive the Szeg{\H o} kernel $s_k$ for $\h^2_k(\mathbb{D}^*)$ directly from the Szeg{\H o} kernel $s$ for $\h^2(\mathbb{D})$. That is,
for $F\in\mathcal{A}_k(\mathbb{D}^*)$ given, let $G\in\mathcal{A}(\mathbb{D})$ such that $F(z)=z^{-k}G(z)$ for $z\in\mathbb{D}^*$. 
Then for $z\in\mathbb{D}^*$, we get
\begin{align*}
  z^kF(z)=G(z)=\int_{b\mathbb{D}}G(w)s(z,w)\;d\sigma_{S^1}(w)
  =\int_{b\mathbb{D}}w^kF(w)s(z,w)\;d\sigma_{S^1}(w).
\end{align*} 
Thus, the kernel  given by
\begin{align*}
  s_k(z,w)=\frac{w^k}{z^k}s(z,w)=\frac{1}{2\pi}\frac{w^k}{z^k(1-z\overline{w})}=\frac{1}{2\pi}\frac{1}{(z\overline{w})^k(1-z\overline{w})}
\end{align*}
exhibits the reproducing property for $\mathfrak{H}^2_k(\mathbb{D}^*)$, and  $\overline{s_k(z,.)}\in\mathfrak{H}^2_k(\mathbb{D}^*)$
for all $z\in\mathbb{D}^*$.
Hence $s_k $ is the Szeg{\H o} kernel for $\mathfrak{H}^2_k(\mathbb{D}^*)$.

Lastly, we remark that $\h^2_k(\mathbb{D}^*)$ satisfies the minimum criteria, laid out in Section~\ref{S:Intro}, for a space $\h$ to be called a reproducing kernel Hilbert space. Here $\mathcal{A}$ corresponds to $\mathcal{A}_k(\mathbb{D}^*)$, while $\mathfrak{X}$ is the space $\mathcal{H}^2_k(\mathbb{D}^*)$ consisting of $F\in\mathcal{O}(\mathbb{D}^*)$ which satisfy
\begin{align*}
  \|F\|_{\mathcal{H}^{2}_k(\mathbb{D}^*)}:=\sup_{0<r<1}\left(\frac{r^{2k}}{2\pi}\int_0^{2\pi}\left|F(re^{i\theta})\right|^2\,d\theta\right)^{\frac{1}{2}}<\infty.
\end{align*}
It follows that
\begin{align}\label{E:alternativeH2k}
  \mathcal{H}_k^2(\mathbb{D}^*)=\Bigl\{F\in\mathcal{O}(\mathbb{D}^*):F(z)=(z^{-k}G(z))|_{\mathbb{D}^*}\text{ for some }G\in\mathcal{H}^2(\mathbb{D}) \Bigr\}.
\end{align}
Moreover, the Laurent series for any function in $\mathcal{H}^2_k(\mathbb{D}^*)$ is of the form $\sum_{j=-k}^\infty a_jz^j$ with $\sum_{j=-k}^\infty |a_j|^2<\infty$. 
This implies that $\mathcal{H}_k^2(\mathbb{D}^*)$ is a Hilbert space. Furthermore, pointwise evalution is bounded on $\mathcal{H}^2_k(\mathbb{D}^*)$. This follows from pointwise evalution being
 bounded on $\mathcal{H}^2(\mathbb{D})$, characterization \eqref{E:alternativeH2k}, and the fact that $z|_{\mathbb{D}^*}\neq 0$. Thus, $\mathcal{H}^2_k(\mathbb{D}^*)$ is a reproducing kernel 
 Hilbert space. Finally, 
$\mathcal{H}^2_k(\mathbb{D}^*)$ and $\h^2_k(\mathbb{D}^*)$  can be seen to be isometrically isomorphic, up to a constant factor,  by mapping the $j$-th Laurent series coeffient of $F\in\mathcal{H}_k^2(\mathbb{D}^*)$ to the $j$-th Fourier coefficient of $F|_{b\mathbb{D}}$ for all $j\geq k$.


\section{Hardy spaces on hypersurface-deleted domains}\label{S:varietydeleteddomains}

The construction of the Hardy spaces for $\mathbb{D}^*$ suggests a general inheritance scheme for the construction of Hardy spaces for domains that are obtained by removing  certain complex hypersurfaces from a given domain. As 
is the case of $\mathbb{D}^*$ in Section~\ref{sec_examples}, one starts with a domain $\Omega$ and a boundary measure $\nu$ that together carry their own  Hardy space structure. We henceforth refer to such a pair $(\Om, \nu)$ as a \textit{parent space}.

We detail requirements on the parent space $(\Om,\nu)$ in Subsection~\ref{SS:parentingspace}. In Subsection~\ref{SS:variety},
we describe the class of complex hypersurfaces that will be removed from $\Om$ to produce the so-called hypersurface-deleted domain $\Om^*$. The inheritance scheme is described in Subsection~\ref{SS:inheritance}.


\subsection{Requirements on the parent space}
\label{SS:parentingspace}
We consider a domain $\Omega\Subset\mathbb{C}^n$ equipped with a finite Borel measure $\nu$ on its topological boundary $b\Omega$. We denote the support of $\nu$ by $T$, and set 
	\bes	
		\Omega_T:=\Omega\cup T.
	\ees
We discuss some conditions that allow us to identify reproducing kernel Hilbert spaces of holomorphic functions on $\Omega$ that admit boundary values on $T$ for, at least, a dense subspace.

\begin{definition}\label{D:weaklyadmissible}
Let $(\Om,\nu)$ be as above, and $\mathcal{F}$ be a family of complex-valued functions on $\Om_T$. Then $\mathcal{F}$ is said to be {\em weakly admissible} if and only if
    \begin{itemize}
    \item[(i)] $F|_T \in L^2(\nu)$ for any $F\in \mathcal{F}$, 		
			and

  \item[(ii)] for any compact set $K\subset\Omega$,  there exists a $C_K>0$ such that
    \bes
		\sup\big\{\,|F(z)|:z\in K\,\big\}\leq
			C_K\left\|F|_T\right\|_{L^2(\nu)}\;\;\text{for\ \ all}\ 		\;\;F\in\mathcal{F}.
		\ees
   \end{itemize}
\end{definition}

If we further assume that $\mathcal{F}$ is closed under subtraction, then each element of $\mathcal{F}$ is uniquely determined by its values along $T$. 

We focus on the family of holomorphic functions given by
	\bes\label{E:A-fam}
 		\mathcal{A}(\Omega,\nu):=\mathcal{O}(\Omega)\cap \mathcal{C}(\Omega_T).
	\ees
Note that $ \mathcal{A}(\Omega,\nu)$ is an algebra over $\mathbb C$.  It satisfies condition (i) in Definition~\ref{D:weaklyadmissible} because $\mathcal{C}(T) \subset L^2(\nu)$ whenever $\nu$ is a finite Borel measure.
  
\begin{definition}\label{D:Hardyspace} 
Let $(\Om,\nu)$ be such that $\mathcal{A}(\Omega,\nu)$ is weakly admissible. We define the {\em pre-Hardy space} associated to $(\Om,\nu)$ as
	\bes
		\mathfrak{H}^2(\Om,\nu):=
			\overline{\mathcal{A}(\Omega,\nu)|_T}^{L^2(\nu)},
	\ees
where
  \bes
   \mathcal{A}(\Omega,\nu)|_T:= \big\{\,f: T\to\mathbb C,\ \  f=F|_T\ \ \text{for some}\ F\in \mathcal{A}(\Omega,\nu)\big\}.
	\ees
   \end{definition}

Proposition \ref{P:Riesz-bd}, and the subsequent discussion, justifies the nomenclature introduced in Definition \ref{E:A-fam}. Note that despite the nonstandard terminology, the following proposition is standard in functional analysis.

\begin{Prop}\label{P:Riesz-bd}
Suppose that $\mathcal{A}(\Omega,\nu)$ is weakly admissible. Then for any $z\in\Omega$, there exists a unique bounded linear functional 
 \bes	
	\Ev_z:\  \mathfrak{H}^2(\Om,\nu)\rightarrow \mathbb C
	\ees 
such that
$\Ev_z(F|_T)=F(z)$ for any $F \in \mathcal{A}(\Omega,\nu)$. Furthermore, there exists a unique function $s:\Om\times T\rightarrow \C$ such that 
\begin{itemize}
  \item[(1)] $\overline{s(z,.)}\in\mathfrak{H}^2(\Om,\nu)$ 
  for all $z\in\Omega$,  and

  \item[(2)] $\Ev_z$ and $s(z,.)$ are related through the integral representation given by
	\bes\label{E:repr}
  \Ev_z(f)= \langle\, f(.), \,\overline{s(z,.)}\, \rangle_{L^2(\nu)}= \int\limits_{T}\!\!f(w) s(z, w)\;d\nu(w)\quad\text{for any}\;\;
  f\in\mathfrak{H}^2(\Om,\nu).
	\ees
  \end{itemize}
 \end{Prop}

We refer to the function $s $ as the  {\em Szeg{\H o} kernel for $\mathfrak{H}^2(\Om,\nu)$}.

\begin{proof}
Note that $\mathcal{A}(\Omega,\nu)|_T$ is a normed vector space when endowed with the norm for $L^2(\nu)$.
The existence of $\Ev_z(f)$ follows from the Bounded Linear Extension Theorem applied to the evaluation $F|_T \mapsto F(z)$ for $F \in \mathcal{A}(\Omega,\nu)|_T$. An application of the Riesz Representation Theorem then yields the existence and uniqueness of $s(z,.)$.
\end{proof}
 
In the literature, Hardy spaces are considered as examples of reproducing kernel Hilbert spaces on $\Om$. Note that $\mathfrak{H}^2(\Om,\nu)$ contains functions that \textit{a priori} are defined only on $T\subseteq b\Om$. 
With an additional assumption on $\mathcal{A}(\Omega,\nu)$, $\mathfrak{H}^2(\Om,\nu)$ may be identified with a function space on $\Omega$, and hence may be considered as a reproducing kernel Hilbert space on $\Omega$. 

To identify the appropriate function space on $\Omega$ for a given weakly admissible $\mathcal{A}(\Omega,\nu)$, we note first that $\Ev_{(.)}(f)$ is holomorphic on $\Omega$ for all $f\in\mathfrak{H}^2(\Omega,\nu)$. This is obvious if there exists an $F\in\mathcal{A}(\Omega,\nu)$ such that $F|_T=f$. It is also true for general $f\in\mathfrak{H}^2(\Omega,\nu)$ because the uniform boundedness of the evaluation operators on compacta, see (ii) in Definition \ref{D:weaklyadmissible}, says that $\Ev_{(.)}(f)$ is the normal limit of holomorphic functions. Thus, the map
\bea\label{E:corresp}
  \mathcal{I}:\mathfrak{H}^2(\Omega,\nu)&
  \longrightarrow& \mathcal{O}(\Omega)\\
  f&\mapsto& F, \textrm{ where }F(z):=\Ev_z(f)\notag
\eea
is well-defined. Denote by $\mathfrak{X}(\Omega,\nu):=   \mathcal{I}\left(\mathfrak{H}^2(\Omega,\nu)\right) \subset \mathcal{O}(\Omega)$. The injectivity of $\mathcal{I}$ can be stated through a condition on certain Cauchy sequences in $\mathcal{A}(\Omega,\nu)$. We formulate this condition for general function spaces as follows. 


\begin{definition}\label{D:stronglyadmissible}
Let $(\Om,\nu)$ be as above, and $\mathcal{F}$ be a weakly admissible family of complex-valued functions on $\Om_T$. Then $\mathcal{F}$ is said to be {\em strongly admissible} if for 
  any sequence $\{F_n\}_{n\in\mathbb{N}}\subset\mathcal{F}$ for which $\left\{(F_n){|_T}\right\}_{n\in\mathbb{N}}$ is Cauchy in $L^2(\nu)$ and $F_n\rightarrow 0$
  uniformly on compacta in $\Om$ as $n\to\infty$, the sequence $\left\{(F_n){|_T}\right\}_{n\in\mathbb{N}}$ converges to $0$ in $L^2(\nu)$ as $n\to\infty$.
\end{definition}

Now suppose $(\Om,\nu)$ is such that $\mathcal{A}(\Omega,\nu)$ is strongly admissible. Then we may equip $\mathfrak{X}(\Omega,\nu)$ with a reproducing kernel Hilbert space structure via $\mathcal I$. This allows us to identify $\h^2(\Om,\nu)$ with a reproducing kernel Hilbert space on $\Omega$, and hence we can make the following definition.

\begin{definition}\label{D:Hardyspace2} Let $(\Om,\nu)$ be such that $\A(\Om,\nu)$ is strongly admissible. The {\em Hardy space} of $(\Om,\nu)$ is $\h^2(\Om,\nu)$. 
\end{definition}

We note that we do not have an independent description of $\mathfrak X(\Om,\nu)$ in this general setting of strongly admissible function spaces. However, in all the examples considered in this paper, $\mathfrak X(\Om,\nu)$ is independently described using an exhaustion-based approach, see the spaces denoted by $\mathcal H^2(.)$ in Sections~\ref{sec_examples}, \ref{S:Stabilization} and \ref{S:Hartogs}. 

\medskip
Examples of $(\Om,\nu)$ for which $\A(\Om,\nu)$ is strongly admissible include
\begin{itemize}
  \item[(1)] $(\Om,\sigma)$, where $\Om\subset\C$ is a $\cont^{1,\alpha}$-smooth bounded domain, and $\sigma$ is the arc-length measure on $b\Om$, see the discussion at the beginning of Section~\ref{S:planar}.
  \item[(2)] $(\mathbb{D}^n,\sigma_{S^1}\times...\times\sigma_{S^1})$, where $\sigma_{S^1}$ is the arc-length measures of the unit circle in the $j$-th coordinate, and $T=(b\mathbb{D})^n$, and
  \item[(3)] $(\Omega,\sigma)$, where $\Omega\subset\Cn$ is a $\cont^2$-smooth bounded domain, $\sigma$ is the surface measure of $b\Omega$ , and $T=b\Omega$, see \cite{St15}.
\end{itemize}
On the other hand, recall from Subsection~\ref{sec_D*} that 
$\A(\mathbb{D}^*,\sigma_{S^1})$ is not even a weakly admissible subspace of $L^2(b\D,\sigma_{S^1})$. Conditions analogous to weak and strong admissibility, albeit in a broader context, were identified in \cite[Theorem p.~347]{A50}. An example is also given therein to demonstrate the inequivalence of the two conditions, see \cite[p.~349]{A50}.
\medskip


\subsection{Requirements on the hypersurface}\label{SS:variety}

We first recall some standard notions from analytic geometry. Let $K\Subset\mathbb{C}^n$ be a bounded set. 

\begin{definition} Denote by $\hol(K)$ the set of equivalence classes of 
	\bes
		\left\{(f,\om):\text{$\om$ is an open neighborhood of $K$ and $f:\om\rightarrow\C$ is holomorphic}\right\}
	\ees
modulo the equivalence relation $(f_1,\om_1)\sim (f_2,\om_2)$ if and only if there is an open neighborhood $\om\subset\om_1\cap\om_2$ of $K$ such that $f_1|_{\om}=f_2|_{\om}$. The equivalence class of $(f,\om)$ will be denoted simply by $f$, which we call the {\em germ of an analytic function on $K$}. Note that $\hol(K)$ forms a ring under multiplication and addition. 
\end{definition}

\begin{definition} Let $\om\subset\Cn$ be an open set. A closed subset $V$ of $\om$ is an \emph{analytic variety in $\omega$} if for any $z\in\omega$ there exists a neighborhood 
  $U(z)\subset\omega$ such that $U(z)\cap V$ is the common zero set of some nontrivial $f_1,\dots, f_k\in\mathcal{O}(U(z))$ for some $k\in\N$. We say that $V$ is a \emph{locally principal variety in $\omega$} if $k$ may be chosen equal to 1 for any $z\in\omega$.
\end{definition}

\begin{definition}
Define $\mathscr V(K)$ to be the set of equivalence classes of 
	\bes
		\left\{(V,\om):\text{$\om$ is an open neighborhood of $K$, $V\subsetneq \om$ is a locally principal variety in $\om$}\right\}
	\ees
modulo the equivalence relation $(V_1,\om_1)\sim (V_2,\om_2)$ if and only if there is an open neighborhood $\om\subset\om_1\cap\om_2$ of $K$ such that $V_1|_{\om}=V_2|_{\om}$. The equivalence class of $(V,\om)$ will be denoted simply by $V$, which we call the {\em germ of an analytic hypersurface in $K$}. 
\end{definition}

We next focus on the situation when $K=\overline{\Omega}$ for
$\Omega\Subset\mathbb{C}^n$ is a domain. Note that the zero set of any  nontrivial $f\in\hol(\overline\Omega)$ gives rise to an element $V\in\mathscr V(\overline\Omega)$, but not every element in $\mathscr V(\overline\Omega)$ arises this way. If $V\in\mathscr V(\overline\Omega)$ is indeed the zero set of a single $f\in\hol(\overline\Omega)$, then $V$ is called {\em principal} and such an $f$ {\em a defining function for $V$}. A principal germ $V$ is called {\em minimally defined} if it admits a defining function $f\in\hol(\overline\Om)$ such that, whenever $U\subset\overline\Om$ is an open set (in the relative topology) and $g\in\hol(U)$ vanishes on $U\cap V$, then $f|_U$ divides $g$ in $\hol(U)$. 
We call such an $f$ a {\em minimal defining function of $V$ in $\hol(\overline\Omega)$}. It follows from a standard argument that 
 minimal defining functions are unique up to non-vanishing holomorphic factors. We state this as a lemma for easy reference. 

\begin{lemma}\label{L:minimality} Let $V$ be a minimally defined germ of an analytic hypersurface in $\overline\Om$. Suppose $f , g\in\hol(\overline\Omega)$ are two minimal defining functions of $V$. Then there is an $h\in\hol(\overline\Omega)$ such that $f =h g$, and $h$ does not vanish on $\overline\Om$. 
\end{lemma}

Finally, $V\in\mathscr V(\overline\Omega)$ is said to be {\em irreducible} if it cannot be expressed as $V_1\cup V_2$ for elements $V_1,V_2\in \mathscr V(\overline\Omega)$ distinct from $V$. Note that for any $V\in\mathscr V(\overline\Om)$, there is an $m\in\N$ such that $V\cap\Om=\cup_{j=1}^m (V_j\cap\Om)$, where each $V_j$ is an irreducible germ of an analytic hypersurface in $\overline\Om$, see \cite[\S~5.4]{Ch12}.

Subsequently, we consider domains as follows. 

\begin{definition}\label{D:vardeldom} Let $\Om\Subset\Cn$ be a domain. Let $V\in\mathscr V(\overline\Omega)$ be a finite union of irreducible, minimally defined germs of  analytic hypersurfaces on $\overline\Om$. Then
	\bes
	\Om^* = \Om\setminus V
	\ees
is called a {\em hypersurface-deleted domain}. 
\end{definition}

We now discuss some examples of hypersurface-deleted domains. In the planar case, if $\Om\Subset\C$ is a domain and $V\subset\mathscr{V}(\overline\Om$), then $\Om\cap V=\{a_1,...,a_m\}$ for some $a_1,...,a_m\in\Om$ and $m\in\N$. It is immediate to see that  $f_j(z)=z-a_j$ is a minimal defining function of $\{a_j\}$ in $\hol(\overline\Omega)$. Thus, $\Om\setminus V=\Om\setminus\{a_1,...,a_m\}$ is a hypersurface-deleted domain. 

A further class of examples, which includes bounded convex domains in $\Cn$, is provided by the following result. Note that the result implies that for such $\Om$, $\Om\setminus V$ is a hypersurface-deleted domain for any $V\in\mathscr V(\overline\Om)$. 

\begin{Prop}\label{P:cousin}
Let $n>1$. Suppose $\Om\Subset\Cn$ is a domain such that $\overline\Om$ admits a Stein neighborhood basis and  $H^2(\overline\Om;\Z)=0$. Then any irreducible $V\in\mathscr V(\overline\Omega)$ is minimally defined. 
\end{Prop}

\begin{proof} The proof is  well-known.  For the reader's convenience, we highlight the main steps of the argument. Recall that a Cousin II distribution on the compact set $\overline\Om$ is a collection $\{(U_\iota,f_\iota)\}_{\iota\in I}$, where $\{U_\iota\}_{\iota\in I}$ is a (relatively) open cover of $\overline\Om$, and $f_\iota\in\hol(U_\iota)$ with $f_\iota|_{U_\iota\cap U_{\jmath}}=h_{\iota\jmath}\cdot f_{\jmath}|_{U_\iota\cap U_{\jmath}}$ for some nonvanishing $h_{\iota\jmath}\in\hol(U_\iota\cap U_{\jmath})$. The hypothesis on $\overline\Om$ implies that, given such a Cousin II distribution, there is an $f\in\hol(\overline\Om)$ such that $f_\iota=h_\iota\cdot f|_{U_\iota}$ for some nonvanishing $h_\iota\in\hol(U_\iota)$, for all $\iota\in I$, i.e., $\overline\Om$ is a Cousin II set, see \cite{Da74}. 

Let $V\in\mathscr V(\overline\Omega)$ be irreducible. Then $V$ admits a local minimal defining function at each point of $V\cap\overline\Om$, see \cite[\S 2.8.]{Ch12}. By compactness and Lemma \ref{L:minimality},
 there is a finite Cousin II distribution, $\{U_i,f_{i}\}_{i\in\{1,...,m\}}$, such that $f_i$ is a minimal defining function of $V\cap U_i$ for $i\in\{1,...,m\}$. We claim that the Cousin II solution, $f\in\hol(\overline\Om)$, for this distribution is a minimal defining function of $V$ in $\hol(\overline\Om)$. First observe that $f|_{U_i\cap V}=(h_i^{-1}\cdot f_i)|_{U_i\cap V}=0$ for $i\in\{1,...,m\}$. Thus, $f$ vanishes on $V$. Next, let $U\subset\overline\Om$ be a (relatively) open subset and $g\in\hol(U)$ be such that $g$ vanishes on $U\cap V$. Since each $f_i$ is minimal, it follows that each $f_i$ divides $g$ in $\hol(U\cap U_i)$. Furthermore,  $f|_{U_i}$ divides $f_i$ in $\hol(U_i)$, in particular $f|_{U\cap U_i}$ divides $f_i$ in $\hol(U\cap U_i)$
 for each $i$. Therefore $f|_{U\cap U_i}$ divides $g$ in $\hol(U\cap U_i)$.
 That is,
 $f|_U$ divides $g$ locally and hence in $\hol(U)$ since $f$ and $g$ are globally defined in $U$. 
\end{proof}

In general, if $V\in\mathscr V(\overline\Omega)$ is principal, then any defining function $f\in\hol(\overline\Omega)$ of $V$ is minimal if and only if $\{z\in\om:\det Df(z)=0\}$ is nowhere dense in $V\cap\om$ for some open neighborhood $\om$ of $\overline\Om$, see \cite[\S~2.9]{Ch12}. Thus, by this criterion, $\Om\setminus V$, where $V$ is an affine hyperplane, is always a hypersurface-deleted domain. 


\subsection{The inheritance scheme}\label{SS:inheritance}
We first construct Hardy spaces for triples of the form $(\Om,\nu, V)$, such that
\begin{itemize}
  \item[(i)] $\Omega\Subset\mathbb{C}^n$ is a domain, $\nu$ is a finite Borel measure on $b\Omega$,
  \item[(ii)] $V$ is an irreducible, minimally defined germ of an analytic hypersurface in $\overline\Omega$, and
  \item[(iii)] $\Om\cap V\neq \emptyset$ and $\nu(T\cap V)=0$, where $T=\supp(\nu)$.
\end{itemize}
The case of general hypersurface-deleted domains is discussed at the end of this subsection. 

As before, $\Om_T=\Om\cup T$, $\Omega^*=\Omega\setminus V$, and $\mathcal{A}(\Omega,\nu)$ is as in Definition~\ref{E:A-fam}. We also set $T^*:=T\setminus V$. Let 
$\psi \in \mathcal{O}(\overline{\Omega})$ be a minimal defining function of $V$. Then for any non-negative integer $k$, we consider the following subset of $\hol(\Om^*)\cap\cont(\Om^*\cup T^*)$
	\begin{align}\label{D:definitionAk}
  	\mathcal{A}_k(\Omega^*,\nu):=\left\{F:\Om^*\cup T^*\rightarrow\C:
	F=(\psi^{-k}G)|_{\Om^*\cup T^*}\
	 \text{for some}\ G\in\mathcal{A}(\Omega,\nu)\qquad  \right. \\
		\left.  \text{and}\ F|_{T^*}\in L^2(\nu)\right\}.\notag
	\end{align}
Note that it follows from Lemma~\ref{L:minimality}, that $\A_k(\Om^*,\nu)$ does not depend on the choice of minimal defining function of $V$. Hence, we make no reference to $\psi$ in our notation and work with a fixed choice of $\psi$ for the purpose of our proofs.

We identify $\A_k(\Om^*,\nu)$ with a function space on $\Om^*\cup T$ by extending its members trivially, by zero, to $T\cap V$, which is a measure-zero set. Then the space of boundary values of $\A_k(\Om^*,\nu)$, i.e.,
	\bes
		\A_k(\Om^*,\nu)|_T=\left\{F|_T: F\in\A_k(\Om^*,\nu)\right\}	
	\ees
is a subspace of $L^2(\nu)$. Note that as subspaces of $L^2(\nu)$, $\A_k(\Om^*,\nu)|_T=\A_k(\Om^*,\nu)|_{T^*}$. This allows us to speak of the notion of weak and strong admissibility for $\A_k(\Om^*,\nu)$. The spaces $\mathcal{A}_k(\Omega^*,\nu)$ always inherit the properties of weak and strong admissibility from $\mathcal{A}(\Omega,\nu)$.

\begin{theorem}\label{T:admissibleheir} For $(\Om,\nu,V)$ satisfying $(i)$, $(ii)$ and $(iii)$ above, the following holds. 
  \begin{itemize} 
     \item[$(1)$]  If $\mathcal{A}(\Omega,\nu)$ is weakly admissible, then so is $\mathcal{A}_k(\Omega^*,\nu)$ for any $k\in\N_0$.
     \item[$(2)$]  If $\mathcal{A}(\Omega,\nu)$ is strongly admissible, then so is $\mathcal{A}_k(\Omega^*,\nu)$ for any $k\in\N_0$.
  \end{itemize}
\end{theorem}

\begin{proof}
  For the proof of part (1), fix a $k\in\N_0$ and suppose that  $\mathcal{A}(\Omega,\nu)$ is weakly admissible. We need to show that for any compact set $K\subset\Omega^*$, there exists
  a constant $c_K>0$ such that the evaluation operators
  \begin{align*}
    \Ev_z:\mathcal{A}_k(\Omega^*,\nu)&\longrightarrow\mathbb{C}\\
    F&\mapsto \Ev_z(F):=F(z),\quad z\in K,
  \end{align*}
  are uniformly bounded on $K$. 
  For that,  let $F\in\mathcal{A}_k(\Omega^*,\nu)$. Then $F=(\psi^{-k}G)|_{\Om^*\cup T^*}$ for some  $G\in\mathcal{A}(\Omega,\nu)$ and $F|_T\in L^2(\nu)$. Since $\mathcal{A}(\Omega,\nu)$ is weakly admissible and $K$ is compact in $\Om^*$, hence in $\Om$, it follows 
  that there exists a constant $C_K>0$ such that
  \begin{align*}
     \left|\Ev_z(G)\right|\leq C_K\left\|G|_T \right\|_{L^2(\nu)}
			\quad \forall z\in K.
  \end{align*}
Therefore, 
	\bes
		\left|\Ev_z(F)\right|=\left|\psi^{-k}(z)\right|\cdot|\Ev_z(G)|\leq C_K\left|\psi^{-k}(z)\right|\left\|G|_T \right\|_{L^2(\nu)}\quad \forall z\in K. 
	\ees
Since $K\subset\Omega^*$, $\psi$ is continuous and nonvanishing on $K$, and $\nu(V\cap T)=0$, there exists a constant $\widetilde{C_K}>0$ such that
  \bes
	\left|\Ev_z(F)\right|\leq\widetilde{C_K}\left\|G|_T \right\|_{L^2(\nu)}
		=\wt{C_k}\left\|(\psi^k\cdot F)|_T \right\|_{L^2(\nu)}
	\quad \forall z\in K.
	\ees
  As $\psi|_T$ is bounded and $F|_T\in L^2(\nu)$, there is a constant $c_K$ such that
  \bes
		\left|\Ev_z(F)\right|\leq c_K\left\|F|_T\right\|_{L^2(\nu)}.
	\ees
  This concludes the proof of part (1).
  
  To prove part (2), let $k\in\N_0$ and suppose that $\mathcal{A}(\Omega,\nu)$ is strongly admissible. Let $\left\{(F_n)\right\}_{n
  \in\mathbb{N}}\subset\mathcal{A}_k(\Omega^*,\nu)$ be a sequence such $\left\{F_n|_T\right\}_{n\in\N}$ is Cauchy in $L^2(\nu)$ and $F_n\longrightarrow 0$ uniformly on compacta in $\Omega^*$. Then for any $n\in\mathbb N$,
  $F_n=(\psi^{-k}G_n)|_{\Om^*\cup T^*}$ for some $G_n\in
  \mathcal{A}(\Omega,\nu)$. Therefore,
  \begin{align*}
    \left\|\left(G_n-G_m \right)|_T \right\|_{L^2(\nu)}=\left\|(\psi^k\cdot F_n-\psi^k\cdot F_m)|_T \right\|_{L^2(\nu)}.
  \end{align*}
  Since $\psi$ is bounded on $T$, it follows that  $\left\{(G_n)|_T\right\}_{n\in\mathbb{N}}$ is a Cauchy sequence in $L^2(\nu)$. Furthermore, $\mathcal{A}(\Omega,\nu)$
  is weakly admissible, and so 
  for any compact set $K\subset\Omega$, there exists  a constant $C_K>0$ such that 
  \begin{align*}
    \left|G_n(z)-G_m(z)\right|\leq C_K\left\|(G_n-G_m)|_T \right\|_{L^2(\nu)}\;\;\forall z\in K,
  \end{align*}
  i.e., $\{G_n\}_{n\in\mathbb{N}}$ converges uniformly on compacta in $\Omega$. Thus, there exists a $G\in\mathcal{O}(\Omega)$ such that $G_n(z)\longrightarrow G(z)$ for all $z\in\Omega$ as $n\to\infty$. 
  However,
  for $z\in\Omega^*$, $G_n(z)=\psi^k(z)F_n(z)\longrightarrow 0$ as $n\to\infty$. Therefore, $G(z)=0$ for all $z\in \Omega^*$. 
  This implies that $G\equiv0$ on $\Omega$, and $G_n\longrightarrow 0$ uniformly on compacta in $\Omega$. 
  Since $\mathcal{A}(\Omega,\nu)$ is strongly admissible, 
  it follows that $(G_n)|_T\longrightarrow 0$ 
  in $L^2(\nu)$ as $n\to\infty$. This in turn implies that $(F_n)|_T\longrightarrow 0$ in $L^2(\nu)$ as $n\to\infty$. Thus, $\A_k(\Om^*,\nu)$ is strongly admissible. 
\end{proof}

We are now set to define the central objects of this discussion.

\begin{definition}
  Let $(\Om,\nu)$ be such that $\mathcal{A}(\Omega,\nu)$ is weakly admissible and $k\in\N_0$. {\em The $k$-th pre-Hardy space} $\mathfrak{H}^2_k(\Omega^*,\nu)$ is the closure of $\mathcal{A}_k(\Omega^*,\nu)|_T$ in $L^2(\nu)$. If $\A(\Om,\nu)$ is strongly admissible, we call $\h^2_k(\Om^*,\nu)$ the {\em $k$-th Hardy space} of $(\Om,\nu,V)$. 
\end{definition}

Note that 
	\bes
		\mathcal{A}_0(\Omega^*,\nu)=\mathcal{A}(\Omega,\nu)|_{\Om^*\cup T^*},
	\ees
i.e., $\A_0(\Om^*,\nu)$ does not lead to a new space. Furthermore,
\be\label{E:monotonicityA}
\mathcal{A}_{0}(\Omega^*,\nu)\subseteq\mathcal{A}_{1}(\Omega^*,\nu)\subseteq\ldots\subseteq\mathcal{A}_k(\Omega^*,\nu)\subseteq\ldots, 
\ee 
 and, for any $\ell\in\N_0$, the spaces
  $\psi^\ell\mathcal{A}_k(\Omega^*,\nu):=\left\{\psi^\ell\cdot F:F\in\mathcal{A}_k(\Omega^*,\nu)\right\}$ satisfy the inclusions
\be\label{E:pseudomonotonicityA}    \psi^{\ell}\mathcal{A}_k(\Omega^*,\nu)\subseteq\mathcal{A}_{k-\ell}(\Omega^*,\nu)\quad
  \text{whenever}\quad \ell\leq k.
\ee
The collection
$\{\mathfrak{H}_k^2(\Omega^*,\nu)\}_{k}$ inherits these properties. That is, $\mathfrak{H}_0^2(\Omega^*,\nu)=\mathfrak{H}^2(\Om,\nu)$. Furthermore,  
\begin{align}\label{E:monotonicityH}
 \mathfrak{H}_0^2(\Omega^*,\nu)\subseteq\mathfrak{H}_1^2(\Omega^*,\nu)\subseteq\dots\subseteq\mathfrak{H}_k^2(\Omega^*,\nu)\dots,
 \end{align}
as well as
\begin{align}\label{E:pseudomonotonicityH}  \psi^\ell\mathfrak{H}_k^2(\Omega^*,\nu)\subseteq\mathfrak{H}^2_{k-\ell}(\Omega^*,\nu)\quad \text{whenever}\quad  \ell\leq k.
\end{align}

Applying Proposition \ref{P:Riesz-bd} to $\mathcal{A}_k(\Omega^*,\nu)$, we see that $\mathfrak{H}^2_k(\Omega^*,\nu)$ possesses a Szeg\H o kernel $s_k $ for any $k\in\mathbb N_0$. Moreover,
the Szeg\H o kernel $s $ for
 $\mathfrak{H}^2(\Om,\nu)$ generates a family of kernels with the reproducing property for $\mathfrak{H}^2_k(\Omega^*,\nu)$.

\begin{Prop}\label{P:kernel}
     Let $(\Om,\nu)$ be such that $\A(\Omega,\nu)$ is weakly admissible. Let $\varphi\in \A(\Omega,\nu)$ be
    such that $\varphi=h\psi$ where $\psi$ is a minimal defining function of $V$ and $h\in\A(\Om,\nu)$ is nonvanishing on $\Om_T\setminus V$. Then 
     \begin{align}\label{E:ckdef}
      c_{k,\varphi}(z,w)
      :=\frac{\varphi^k(w)}{\varphi^k(z)}\,s(z,w)
      ,\;\;z\in\Omega^*,w\in T
      \end{align}
       has the reproducing property for $\mathfrak{H}^2_k(\Omega^*,\nu)$.
      Moreover, if $h$ is nowhere vanishing on $\Om_T$ and $|\varphi|$ is constant on $T$, then $c_{k,\varphi} $ is the Szeg{\H o} kernel for $\mathfrak{H}^2_k(\Omega^*,\nu)$ for all $k\in\N_0$. 
\end{Prop}    

\begin{proof} Let $k\in\N_0$ and $F\in\mathcal{A}_k(\Omega^*,\nu)$. Then there is a $G\in\A(\Om,\nu)$ such that $F=(\psi^{-k}G)|_{\Om^*\cup T^*}$ and $F|_T\in L^2(\nu)$. Since $h^kG\in\A(\Om,\nu)$, it follows that for any $z\in\Om^*$  
\begin{align*}
    F(z)=\varphi^{-k}(z)\varphi^{k}(z)F(z)=\varphi^{-k}(z)h^k(z)G(z)
		=\varphi^{-k}(z)\int\limits_{T}\!\!h^k(w)G(w)\,s(z,w)\;d\nu(w).
  \end{align*}
Thus, 
	\bes
  F(z)=
	\int\limits_T F(w)\,c_{k, \varphi}(z,w)\, d\nu(w)\quad \text{for any}\ z\in\Om^*.
  \ees
  The reproducing property of $c_{k,\varphi} $ for $\h^2_k(\Om^*,\nu)$ then follows from the density of $\mathcal{A}_k(\Omega^*,\nu)|_T$ in $\mathfrak{H}^2_{k}(\Omega^*,\nu)$
  with respect to  $L^2(\nu)$.
  
 It remains to show that if $h$ is nonvanishing on $\Om_T$ and $|\varphi|$ equals some constant $c\geq 0$ on $T$, then $\overline{c_{k, \varphi}(z,.)}\in \mathfrak{H}^2_k(\Omega^*,\nu)$ for any $z\in\Om^*$. Note first that $c\neq 0$ since neither $h$ nor $\psi$ vanish on $T^*$. Thus, as an aside, observe that $\varphi$ does not vanish on $T$ and, in particular, $V\cap T=\emptyset$.
  Since $\overline{s(z,.)}\in \mathfrak{H}_k^2(\Omega,\nu)$
   for any $z\in\Omega$, it follows that
   there exists a sequence
    $\{S_n(z,.)\}_{n\in\N}$ such that $\overline{S_n(z,.)}\in \mathcal{A}(\Omega,\nu)$ for all $n\in\N$ and  
   \bes
		\big\|s(z,.)-S_n(z,.)|_T\big\|_{L^2(\nu)}\longrightarrow 0\;\;\text{as}\;\;n\to\infty,\;\;\forall\;z\in\Omega.
	\ees
This, and the fact that $\varphi^k(.)\varphi^{-k}(z)$ is bounded on $T$ for any fixed $z\in\Om^*$, implies that
	\bes
	 \big\|c_{k,\varphi}(z,.)-\left(\varphi^k(.)\varphi^{-k}(z)S_n(z,.)\right)|_T
		\big\|_{L^2(\nu)}
	\longrightarrow 0\;\;\text{as}\;\;n\to\infty,\;\;\forall\;z\in\Omega^*.
	\ees
To see that $\overline{\varphi^k(.)\varphi^{-k}(z)S_n(z,.)}$ is in $\A_k(\Om^*,\nu)$ for any $z\in\Om^*$, we first note $$\overline{\varphi^k(w)} = c^{2k}\varphi^{-k}(w)\;\forall\;w\in T.$$ It then suffices to show that $\varphi^{-k}(.)\overline{S_n(z,.)}$ is in $\A_k(\Om^*,\nu)$ for any $z\in\Om^*$. Since $h\in~\A(\Om,\nu)$ is nonvanishing on $\Om_T$, it follows that $h^{-k}(.)\overline{S_n(z,.)}\in\A(\Om,\nu)$. Thus, by the definition of $\A_k(\Om^*,\nu)$, it remains to show that  $\left(\psi^{-k}(.)h^{-k}(.)\overline{S_n(z,.)}\right)|_T$ is in $L^2(\nu)$. This membership holds because $\psi\cdot h=\varphi$ is a nonvanishing continuous function on $T$.  This concludes the proof of $c_{k,\varphi} $ being the Szeg{\H o} kernel for $\h^2_k(\Om^*,\nu)$. 
\end{proof} 

\begin{remark}\label{R:repker}
Note that replacing the Szeg\H o kernel for $\mathfrak{H}^2(\Om,\nu)$ in \eqref{E:ckdef} with {\em any} other kernel with the reproducing property for $\mathfrak{H}^2(\Om,\nu)$, yields yet another family of kernels with the reproducing property for $\mathfrak{H}^2_k(\Omega^*,\nu)$.
\end{remark}

We briefly discuss the especially favorable situation when $V\cap T=\emptyset$.
In this case the requirement that $F|_{T^*}\in L^2(\nu)$ in the definition of $\A_k(\Om^*,\nu)$ is redundant. Moreover, the containment relations in \eqref{E:monotonicityA} and \eqref{E:monotonicityH} are strict, i.e.,
\bes
	\mathcal{A}_\ell(\Omega^*,\nu) \subsetneq \mathcal{A}_{k}(\Omega^*,\nu)\quad\text{and}\quad 
   \mathfrak{H}^2_\ell(\Omega^*,\nu)\subsetneq  \mathfrak{H}^2_{k}(\Omega^*,\nu),
  \ \ \text{when}\: \ell\leq k,\ \ell\in\N_0,
\ees
and those in \eqref{E:pseudomonotonicityA} and 
\eqref{E:pseudomonotonicityH} are equalities, i.e.,
\bes 
  \psi^{\ell}\mathcal{A}_k(\Omega^*,\nu) = \mathcal{A}_{k-\ell}(\Omega^*,\nu)\quad\text{and}\quad 
   \psi^{\ell}\mathfrak{H}^2_k(\Omega^*,\nu) = \mathfrak{H}^2_{k-\ell}(\Omega^*,\nu),
  \ \ \text{when}\: \ell\leq k,\ \ell\in\N_0.
\ees
Theorem \ref{thm_eggs} provides examples of $(\Om, \nu, V)$ that exhibit the dual phenomenon,  i.e., the containments \eqref{E:monotonicityH} stabilize to equalities, while the containments in \eqref{E:pseudomonotonicityH} are strict. 

In the classical construction, the Hardy space $\mathfrak{H}^2(\Omega, \nu)$ is a module over the algebra $\mathcal{A}(\Omega,\nu)$. This phenomenon cannot percolate to $\h^2_k(\Omega^*,\nu)$ as, in general, $\mathcal{A}_k(\Omega^*,\nu)$ is not even an algebra. However, when $V\cap T=\emptyset$, the union
$\bigcup_{k=0}^\infty\mathcal{A}_k(\Omega^*,\nu)$
is a filtered algebra over $\C$ since 
$$\mathcal{A}_k(\Omega^*,\nu)\cdot\mathcal{A}_j(\Omega^*,\nu)	
\subseteq\mathcal{A}_{k+j}(\Omega^*,\nu),\quad j,k\in\N_0.$$
The space $\bigcup_{k=0}^\infty\mathfrak{H}_k^2(\Omega^*,\nu)$ is then a filtered module over this filtered algebra since 
$$\mathcal{A}_k(\Omega^*,\nu)\cdot\h^2_j(\Omega^*,\nu)	
\subseteq\h^2_{k+j}(\Omega^*,\nu),\quad j,k\in\N_0.$$

We now consider the general case, i.e., $V=V_1\cup \dots\cup V_m$, where each $V_j$ is an irreducible, minimally defined germ of an analytic hypersurface in $\overline\Om$. Let $\psi_j\in\mathcal{O}(\overline{\Omega})$ be a minimal defining function of $V_j$, $j\in\{1,\dots,m\}$. Then $\psi=\psi_1\cdot\ldots\cdot\psi_m\in \mathcal{O}(\overline{\Om})$ is a minimal defining function of $V$. One could proceed as in Definition~\ref{D:definitionAk} using $\psi$. However, this approach leads to an incomplete picture of the relevant spaces as each irreducible germ can independently yield a one-parameter family of spaces. For instance, consider the example $\Omega^*_P$ at the beginning of Section~\ref{S:planar}, and compare the spaces in \eqref{E:Akplanar} to the above definition where all the factors of $\psi$ would appear with the same exponent.

To remedy this issue we proceed inductively. We write $$\Omega^*_\ell=\Omega\setminus \left(V_1\cup\ldots\cup V_\ell\right),
	\quad  \ell\in\{1,\dots,m\},$$
and define $\A_{\bk}(\Om^*_1,\nu)$ as in Definition~\ref{D:definitionAk} for $\bk\in\N_0$. For $\ell\geq 2$, consider multi-indices
$\bk=\{k_1,\dots,k_\ell\}$ and $\bk'=\{k_1,\dots,k_{\ell-1}\}$ with $k_j\in\mathbb{N}_0$, and define
\beas
	\quad 
  \mathcal{A}_{\bk}(\Omega^*_\ell,\nu)
		:=\left\{ F:\Om^*\cup T^*\rightarrow\C:
		F=(\psi_{\ell}^{-k_\ell}G)|_{\Om^*\cup T^*}\ \text{for some}\  
  			G\in\mathcal{A}_{\bk'}(\Omega^*_{\ell-1},\nu)\quad\right.\\
		\left.\text{and}\ F|_{T^*}\in L^2(\nu)
		\right\}.
\eeas
The inductive nature of this definition allows for the iterated application of
 Theorem \ref{T:admissibleheir} and Proposition \ref{P:kernel}. In particular, if $(\Om,\nu)$ is such that $\A(\Om,\nu)$ is strongly admissible, then 
\begin{equation}\label{E:Hardy-ind}
\mathfrak{H}_{\bk}^2(\Omega^*,\nu):= \overline{\mathcal{A}_{\bk}(\Omega_m^*,\nu)|_T}^{\,L^2(\nu)}
\end{equation}
is a reproducing kernel Hilbert space on $\Om$ for any $\bk=\N_0^m$, and we call it the {\em $\bk$-th Hardy space} of $(\Om,\nu,V)$.

\section{Planar domains}\label{S:planar}

In this section, we apply the scheme described in Subsection~\ref{SS:inheritance} to hypersurface-deleted planar domains. Note that the case of the punctured disk is covered in Subsection~\ref{sec_D*}. 

Recall that any hypersurface-deleted planar domain may  be written as 
	\bes
		\Om^*_P=\Om\setminus P,\quad P=\{p_1,...,p_m\}\subset\Om,
	\ees
see Definition~\ref{D:vardeldom} and the subsequent discussion. We henceforth refer to $\Om_P^*$ as an {\em $m$-punctured domain}. Here, we consider $\Om\Subset\C$ of class $\cont^{1,\alpha}$ for $\alpha\in(0,1)$, and the arc-length measure $\sigma$ on $b\Om$ so that $V\cap \supp(\sigma)=\emptyset$. Under these assumptions, $\A(\Om,\sigma)$ is strongly admissible. This is because $\A(\Om,\sigma)\subset E^2(\Om)$, the classical Smirnov--Hardy space of $\Om$, which is strongly admissible due to the existence of nontangential limits in $L^2(\sigma)$, see \cite[Theorem 10.3 \& Section 10.5]{Du70}. In fact, $\A(\Om,\sigma)|_{b\Om}$ is dense in $E^2(\Om)|_{b\Om}$, see \cite[Theorem 10.6 \& Section 10.5]{Du70}. Thus, the Hardy space $\h^2(\Om,\sigma)$ coincides with the classical Hardy space on $\Om$. Now we can either apply the inductive scheme of Section~\ref{SS:inheritance} or, equivalently, consider the closure in $L^2(\sigma)$ of the strongly admissible space of boundary values of 
	\be\label{E:Akplanar}
		\A_{\bk}(\Om^{*}_P,\sigma)=\{F\in\hol(\Om^{*}_P):(z-p_1)^{k_1}\cdot...\cdot(z-p_m)^{k_m}
F(z)\in\A(\Om,\sigma)\}
	\ee
	for $\bk=(k_1,...,k_m)\in\N_0^m$.
Either construction gives
 a family of Hardy spaces $\left\{\h^2_{\bk}(\Om^{*}_P,\sigma)\right\}_{\bk\in\N_0^m}$ such that 
		$$\h^2_{\bk}(\Om^{*}_P,\sigma)\subsetneq \h^2_{\bk'}(\Om^{*}_P,\sigma)\quad
		\text{whenever} \quad k_j\leq {k_j}',\ \ j\in\{1,...,m\}.$$
 Note that each $\h^2_{\bk}(\Om^{*}_P,\sigma)$ is the space of $L^2$-boundary values of holomorphic functions on $\Om$ that have poles of orders at most $k_1,...,k_m$ at $p_1,...,p_m$, respectively. Applying Proposition~\ref{P:kernel} and Remark~\ref{R:repker} iteratively, we obtain the following result. 

\begin{Prop}\label{prop_repker2} Let $\bphi=(\phi_1,...,\phi_m)\in \A(\Om,\sigma)^m$, $\bk=(k_1, \ldots, k_m) \in\N_0^m$ and $\bphi^{\pm\bk}=\phi_1^{\pm k_1}\cdot...\cdot\phi_m^{\pm k_m}$. Suppose each $\phi_j$ vanishes only at $p_j$, $j=1,...,m$. Suppose $c(z,w)$ is a kernel with the reproducing property for $\h^2(\Om,\sigma)$.
Then
 $$\bphi(z)^{-\bk}\,c(z,w)\,\bphi(w)^{\bk},\quad z\in \Om^*_P, w\in b\Om,
$$ 
has the reproducing property for $\h^2_{\bk}(\Om^*_P,\sigma)$. Further, if each $\phi_j$ has a simple zero at $p_j$ and $|\phi_j|$ is constant on $b\Om$, then
	\bes
		c_{\bk,\bphi}(z,w):=\frac{\bphi(w)^{\bk}}{\bphi(z)^{\bk}}s(z,w),
		\quad z\in \Om^*_P, w\in b\Om,
	\ees
is the Szeg{\H o} kernel for  $\h^2_{\bk}(\Om^*_P,\sigma)$ for all $\bk\in\N_0^m$.
\end{Prop}

In addition to the Szeg{\H o} kernel, we discuss a generalization of the Cauchy kernel for $\h^2_{\bk}(\Om^*_P,\sigma)$, $\bk\in\N_0^m$. Recall that the classical Cauchy kernel 
$$ \cont(z,w)=
\frac{1}{2\pi i}\,\frac{1}{w-z}$$ is a holomorphic function on $\C\times\C\setminus\{z=w\}$ such that 
$$\frac{j^*\big(\cont (z,w)dw\big)}{d\sigma(w)}$$
 has the reproducing property for $\h^2(\Om,\sigma)$, where $j:b\Om\rightarrow\C$ is the inclusion map. Applying Proposition~\ref{prop_repker2} to this kernel, we obtain the following analog of the Cauchy integral formula for $m$-times punctured domains
	\bes
		F(z)
		=\frac{1}{2\pi i}\int_{b\Om} \underbrace{\frac{(w-p_1)^{k_1}\cdots (w-p_m)^{k_m}}{(z-p_1)^{k_1}\cdots (z-p_m)^{k_m}(w-z)}}_{=:2\pi i\,\cont_{\bk}(z,w)}F(w)\,dw
	\ees 
for $F\in\A_{\bk}(\Om^*_P,\sigma)$ and $z\in\Om^*_P$.
We call $\cont_{\bk}(z,w)$ the {\em Cauchy $\bk$-kernel for $m$ punctures}. Note that it is a meromorphic function on $\C\times\C\setminus\{z=w\}$ whose poles depend solely on the location of the punctures. When written with respect to $\sigma$, the integral kernel in the above formula is, in fact, 
$$\cont^{\Om^*_P}_{\bk}(z,w):=\cont_{\bk}(z,w)\dot{\gamma}(w),$$
where $w=\gamma(t)$ is the arc-length parametrization of $b\Om$. It follows that $\cont^{\Om^*_P}_{\bk}(z,w)$ has
 the reproducing property for $\h^2_{\bk}(\Om^*_P,\sigma)$. In contrast to $\cont^{\Om^*_P}_{\bk} $, the Szeg{\H o} kernel, $s_{\bk} $, of $\h^2_{\bk}(\Om^*_P,\sigma)$ is, in general, not known explicitly. However, for simply connected $\Om$, Theorem \ref{T:rigidity} below gives a formula for $s_{\bk} $ in terms of the Szeg\H o kernel for  $\h^2(\Om,\sigma)$. It also shows that the two kernels, $s_{\bk}$ and $\cont^{\Om^*_P}_{\bk}$, coincide if and only if $\Om^*_P$ is a disk punctured at its center. This rigidity result extends the Kerzman--Stein Lemma (\cite[Lemma~7.1]{KeSt78}) to the case of $m$-punctured domains.

\begin{theorem}\label{T:rigidity}
 Let $\Om\Subset\C$ be a $\cont^{1,\alpha}$-smooth simply connected domain, and $P=\{p_1,...,p_m\}\subset\Om$. Let $\mu:\Om\rightarrow\D$ be a biholomorphism with
$q_j=\mu(p_j)$, $j=1,...,m$.  
	\begin{enumerate}
		\item The Szeg{\H o} kernel for $\h^2_{\bk}(\Om^*_P,\sigma)$ is given by
			\bes
	s_{\bk}(z,w)=\bphi_0^{-\bk}(z)\,s(z,w)\,\overline{\bphi_0^{-\bk}(w)},
				\quad z\in\Om^*_P, w\in b\Om,
		\ees
where, $\bphi_0=\left(M_{q_1}\circ\mu,...,M_{q_m}\circ\mu\right)$ for $M_q(\zt)=\dfrac{\zt-q}{1-\overline q\zt}$, $(q,\zt)\in\D\times \overline\D$.
		\item $\cont^{\Om^*_P}_{\bk}(z,w)=s_{\bk}(z,w)$ for some $\bk\in\N_0^m$ if and only if $\Om^*_P$ is a disk punctured at its center. 
	\end{enumerate}
\end{theorem}
To prove Theorem~\ref{T:rigidity}, we use the fact that the Szeg{\H o} kernel $s $ of $\h^2(\Om,\sigma)$ is $S |_{\Om\times b\Om}$, where $S $ is the continuous extension of the Szeg{\H o} kernel for $E^2(\Om)$ to $\overline\Om\times\overline\Om\setminus\{(z,z):z\in b\Om\}$. Note that $S(z,w)=\overline{S(w,z)}$ for $z,w\in\overline\Om\times\overline\Om\setminus\{(z,z):z\in b\Om\}$. The continuous extension of the Szeg{\H o} kernel for $E^2(\Om)$ follows from three facts. Firstly, this is true for the classical Szeg{\H o} kernel $S^\D $ of the disk. Secondly, the derivative of any biholomorphism $\beta$ from $\Om$ onto $\D$ admits a continuous nonvanishing square root on $\overline\Om$, see \cite[Theorem 3.5]{Po13}. And lastly, the Szeg{\H o} kernel for $E^2(\Om)$ can be expressed in terms of $S^\D$ and $\sqrt{\beta'}$, see the transformation law in \cite[Lemma 5.3]{La98}.

{\em Proof of Theorem~\ref{T:rigidity}.} Since $\mu$ extends continuously to $\overline\Om$, we have that $\bphi_0\in\A(\Om,\sigma)^m$. Moreover, since $|M_q|\equiv 1$ on $b\D$, and $M_q$ only has a simple zero at $q$, the same is true of each $\phi_j$ on $b\Om$ and at $p$, respectively. Thus, by Proposition~\ref{prop_repker2}, $c_{\bk,\bphi_0}$ is the Szeg{\H o} kernel for $\h^2_{\bk}(\Om^*_P,\sigma)$. Since $M_{q}\circ\mu=\overline{(M_{q}\circ\mu)^{-1}}$ on $b\Om$, the first claim follows.  

Next, observe that $S_{\bk}(z,w)
=\bphi_0(z)^{-\bk}S(z,w)\,\overline{\bphi_0(w)^{-\bk}}$ extends $s_{\bk}$ continuously to $(\overline\Om\setminus P)^2\setminus\{(z,z):z\in b\Om\}$, and $S_{\bk}(z,w)=\overline{S_{\bk}(w,z)}$. Thus, if $\cont_{\bk}^{\Om^*_P}=s_{\bk}$, it must be that for $z,w\in b\Om$, $z\neq w$,
	\be\label{eq_C=S}
		\cont_{\bk}^{\Om^*_P}(S(z,w))-\overline{\cont_{\bk}^{\Om^*_P}(S(z, w))}=
		\frac{1}{2\pi i}\frac{e(S(z,w))}{w-z}\left(\dot\gamma(w)
	-\frac{1}{|e(S(z,w))|^2}\wt{\dot\gamma(z)}\right)=0,
	\ee
where $e(S(z,w))=\dfrac{(w-p_1)^{k_1}\cdots (w-p_m)^{k_m}}{(z-p_1)^{k_1}\cdots (z-p_m)^{k_m}}$, and $\wt{\dot\gamma(z)}=\overline{\dot\gamma(z)}\dfrac{w-z}{\wbar-\zbar}$ is the vector obtained from reflecting $\dot\gamma(z)$ in the chord determined by $w$ and $z$. Thus, as in the proof of the classical Kerzman--Stein Lemma, \eqref{eq_C=S} implies that for any two distinct points $z,w\in b\Om$, the chord connecting $w$ and $z$ meets the boundary curve with the same angle at both points. But this can only happen if $b\Om$ is a circle \cite{RaTo57}, i.e., $\Om =\mathbb D_r(a)=\{z\in\C:|z-a|<r\}$ for some $a\in\C$ and $r>0$. In this case $|\dot\gamma(w)|=|\wt{\dot\gamma(z)}|$ for all $z,w\in b\mathbb D_r(a)$, and so $|e(S(z,w))|\equiv 1$ for $z,w\in b\mathbb D_r(a)$. If $\bk\in\N_0^m$, this yields that $|(w-p_1)\cdots(w-p_m)|$ is constant on $b\mathbb D_r(a)$, which is only possible if $P=\{a\}$. 
\qed

Theorem~\ref{T:rigidity} is stated only for simply connected domains because of the limited applicability of Proposition~\ref{prop_repker2}. In particular, if $\Om$ is multiply connected, then the conditions on $\phi_j$, assumed in Proposition~\ref{prop_repker2}, may not be attainable. For example, if $\Om=\{z\in\C:1<|z|<2\}$ and $V=\{a\}$ for some $a\in\Om$, then there is no $\phi\in\A(\Om,\sigma)$ that has a simple zero at $a$ and is such that $|\phi|\equiv C $ on $b\Om$. This is because, owing to the argument and maximum principles, $N(\xi):=\frac{1}{2\pi i}\int_{b\Om}\frac{\phi'(w)}{\phi(w)-\xi}dw$
 is a continuous, integer-valued function on $\D_C(0)$ and hence a constant. If $\phi$ had a simple zero, then $N\equiv 1$ on $\D_{C}(0)$, forcing $\phi$ to be a homeomorphism between $\Om$ and $\D_C(0)$, which is impossible. 
 
However, in the case when $\Om$ is finitely connected, the Szeg{\H o} kernel for $\h^2_{\bk}(\Om^*_P,\sigma)$ enjoys a transformation law under biholomorphisms. The proof goes along classical arguments in \cite[Ch.~12]{Be92} and \cite[Lemma~5.3]{La98}, after taking into account the boundary regularity of conformal maps between $\cont^{1,\alpha}$-smooth domains, see \cite[App. A]{BaFiLe14}.

\begin{theorem} Suppose $\Om,D\Subset\C$ are $\cont^{1,\alpha}$-smooth domains, and $\mu:\Om\rightarrow D$ is a biholomorphism. Then, for $\bk\in\N_0^m$, 
	\bes
		s_{\bk}^{\Om_P^*}(z,w)=\sqrt{\mu'(z)}\, \left(s_{\bk}^{D^*_{\mu(P)}}\left(\mu(z),\mu(w)\right)\right)\,\overline{\sqrt{\mu'(w)}},\quad
		z\in\Om_P^*, w\in b\Om,
	\ees
where $s_{\bk}^{\Om_P^*} $ and $s_{\bk}^{D_{\mu(P)}^*} $ denote the Szeg{\H o} kernels for $\h^2_{\bk}(\Om_P^*,\sigma)$ and $\h^2_{\bk}(D_{\mu(P)}^*,\sigma)$, respectively. 
\end{theorem}


\section{Hypersurface-deleted egg domains as examples of finite stabilization}\label{S:Stabilization}

In this section, we consider triples of the form $\big(\E_{p},\nu,\{z_2=0\}\big)$, $p\in\N$, where 
	\be\label{eq_eggs}
		\E_{p}=\big\{(z_1,z_2)\in\CC:|z_1|^{2p}+|z_2|^{2p}<1\big\},
	\ee
and the measure $\nu$ on $b\E_p$ is either
\begin{itemize}
	\item [$(a)$] $\sigma$, the Euclidean surface area measure, or
	\item [$(b)$] $\om_p$, the Monge--Amp{\` e}re boundary measure associated to the exhaustion function  $$\varphi_p(z_1,z_2)=\frac{1}{2p}\log\left(|z_1|^{2p}+|z_2|^{2p}\right).$$ 
\end{itemize}
Note that $\om_p$ is also the Leray--Levi measure associated to the defining function $$\rho_p(z_1,z_2)=\frac{2\pi}{p}\left(|z_1|^{2p}+|z_2|^{2p}-1\right).$$
In the case of the ball, or $p=1$, the two measures coincide and $\mathfrak{H}^2(\E_1,\sigma)=\mathfrak{H}^2(\E_1,\omega_1)$. In all other cases, $\mathfrak{H}^2(\E_p,\sigma)\subsetneq \mathfrak{H}^2(\E_p,\omega_p)$. We show that this discrepancy, owing to different choices of measure, is amplified in the case of $\E_p^*=\E_p\setminus\{z_2=0\}$. Moreover, this setting yields examples of nontrivially stabilizing filtrations of Hardy spaces.

For some context, note that $\mathfrak{H}^2(\E_p,\sigma)$ is the space of boundary values of the classical Hardy space on $\E_p$ as defined by Stein in~\cite{St15}, while $\mathfrak{H}^2(\E_p,\om_p)$ is the space of boundary values of the Poletsky--Stessin Hardy space associated to $\varphi_p$ on $\E_p$, see \cite{PoSt08}. The latter spaces have been studied by Hansson in \cite{Ha99}, {\c S}ahin in \cite{Sa16}, and Barrett--Lanzani in \cite{BaLa09}. Later, we encounter the limiting case of $\big(\E_p,\om_p,\{z_2=0\}\big)$ as $p\rightarrow\infty$. To wit, if $\E_\infty=\lim\limits_{p\rightarrow\infty}\E_p$ in the Hausdorff metric, and $\om_\infty$ is the Monge--Amp{\` e}re measure corresponding to the function 
\bes
\varphi_\infty(z_1,z_2)=\lim\limits_{p\rightarrow\infty}\varphi_p(z_1,z_2)
	=\log \max\{|z_1|,|z_2|\},
\ees
 then $\E_\infty^*=\E_\infty\setminus\{z_2=0\}$ is $\D\times\D^*$ and $\om_\infty=\sigma_{S^1}\times\sigma_{S^1}$, see \cite[\S~4]{PoSt08}. Since $\{z_2=0\}$ does not intersect $\operatorname{supp}(\om_\infty)=b\D\times b\D$,  $\{\mathfrak{H}^2_k(\E_\infty^*,\omega_\infty)\}_{k\in\N_0}$ does not stabilize, and $z_2^\ell\h^2_k(\E_\infty^*,\omega_\infty)=
			 \mathfrak{H}^2_{k-\ell}(\E_\infty^*,\omega_\infty)$ for $\ell\leq k$. The behavior of this filtration is quite different when $p<\infty$.

\begin{theorem}\label{thm_eggs} Let $p\in\N$. Then $\{\h^2_k(\E_P^*,\sigma)\}_{k\in\N_0}$ stabilizes at $k=0$, i.e., 
	\bes
		\mathfrak{H}_k^2(\E_p^*,\sigma)=\mathfrak{H}_0^2(\E_p^*,\sigma),\quad \forall k\in\N_0.
	\ees
On the other hand, $\{\h^2_k(\E_P^*,\omega_p)\}_{k\in\N_0}$ stabilizes at $k=p-1$, i.e.,
\bes
	\mathfrak{H}^2_0(\E_p^*,\omega_p)\subsetneq
			 \mathfrak{H}^2_1(\E_p^*,\omega_p)\subsetneq\cdots\subsetneq \mathfrak{H}^2_{p-1}(\E_p^*,\omega_p)=\mathfrak{H}^2_{k}(\E_p^*,\omega_p),\quad \forall k\geq p.
\ees
Moreover, $\mathfrak{H}^2_0(\E_p^*,\omega_p)\supsetneq
			 z_2\mathfrak{H}^2_1(\E_p^*,\omega_p)\supsetneq\cdots\supsetneq z_2^{k}\mathfrak{H}_{k}^2(\E_p^*,\omega_p)
				\supsetneq \cdots.$
\end{theorem}

In order to prove Theorem~\ref{thm_eggs}, we describe the relevant Hardy spaces. Here, $j:b\E_p\rightarrow \CC$ denotes the inclusion map, and $d^c$ is the real operator $i(\dbar-\partial)$. Dropping the subscripts of $\varphi_p$ and $\rho_p$, we have that
	\bes
		\om_p=j^*(d^c\varphi\wedge dd^c\varphi)
		=\frac{j^*(\partial\rho\wedge\dbar\partial\rho)}{(2\pi i)^2}
			=\frac{-\det\begin{pmatrix}
						0 & \rho_{\overline{z}_1} &\rho_{\overline{z}_2} \\ 
							\rho_{z_1} & \rho_{z_1\overline{z}_1} & \rho_{z_1\overline{z}_2}\\
								\rho_{z_2} & \rho_{z_2\overline{z}_1} & \rho_{z_2\overline{z}_2}\\
					\end{pmatrix}}{\pi^2|\nabla\rho|}\,\sigma,\quad 
		\text{on}\ b\E_p,
	\ees
where $\rho_{z_j}$ is the first order partial derivative of $\rho$ with respect to $z_j$, $j\in\{1,2\}$, and $\rho_{z_j\overline{z}_k}$ is 	the second order partial derivative of $\rho$ with respect to $z_j$ and $\overline{z}_k$, $j,k\in\{1,2\}$. For ease of computation, we parametrize $(b \E_p)_*:=b\E_p\setminus\{(z_1,z_2)\in\mathbb{C}^2: z_1 z_2=0\}$ as 
	\be\label{eq_coords}
		\vartheta:(s,\theta_1,\theta_2)\mapsto 
			\left(s^\frac{1}{2p}e^{i\theta_1},(1-s)^\frac{1}{2p}e^{i\theta_2}\right),\quad 
				(s,\theta_1,\theta_2)\in (0,1)\times[0,2\pi)^2.
	\ee
Since $b\E_p\cap\{z_1z_2=0\}$ is a set of measure zero for both $\sigma$ and $\om_p$, we have that
	\be\label{eq_L2eggs}
		L^2(b\E_p,\nu)\cong L^2\left((0,1)\times[0,2\pi)^2;\vartheta^*\nu\right),\quad \nu=\sigma,\om_p,
	\ee
via the map $f\mapsto f|_{(b\E_p)_*}\circ \vartheta$. It is easy to check that
	\begin{align*}
		\vartheta^*d\sigma
			&=\dfrac{1}{2p}\dfrac{\sqrt{(s)^{2-\frac{1}{p}}+(1-s)^{2-\frac{1}{p}}}}			{s^{1-\frac{1}{p}}(1-s)^{1-\frac{1}{p}}}\,ds\,d\theta_1\,d\theta_2
					\approx\dfrac{\,ds\,d\theta_1\,d\theta_2}
			{s^{1-\frac{1}{p}}(1-s)^{1-\frac{1}{p}}},\ \text{and}\\
		\vartheta^*\om_p&= ds\,d\theta_1\,d\theta_2.
	\end{align*}
Here, $a(r)\approx b(r)$ means that there are constants $c,C>0$ such that $cb(r)\leq a(r)\leq Cb(r)$ for all $r$. 
For the sake of brevity, we drop all references to $\vartheta$, use $(s,\theta_1,\theta_2)$ as coordinates on $b\E_p$, and abbreviate $||f||_{L^2(b\E_p,\nu)}$ to $||f||_\nu$. We now provide descriptions of the spaces $\mathfrak{H}^2(\E_p,\sigma)$ and $\mathfrak{H}^2(\E_p,\om_p)$ in terms of $L^2$-convergent series expansions. 

\begin{Prop}\label{prop_descHardy} Let $p\in\N$. Then 
\bea
	\qquad \mathfrak{H}^2(\E_p,\sigma)&=&
		\left\{\sum\limits_{j,\ell\geq 0} a_{j,\ell}s^{\frac{j}{2p}}(1-s)^\frac{\ell}{2p}e^{i(j\theta_1+\ell\theta_2)}
		:\sum\limits_{j,\ell\geq 0} |a_{j,\ell}|^2\beta\left(\frac{j+1}{p},\frac{\ell+1}{p}\right)<\infty\right\},
				\label{eq_hardysurf}\\
	\qquad \mathfrak{H}^2(\E_p,\om_p)&=&
		\left\{\sum\limits_{j,\ell\geq 0} a_{j,\ell}s^{\frac{j}{2p}}(1-s)^\frac{\ell}{2p} e^{i(j\theta_1+\ell\theta_2)}
		:\sum\limits_{j,\ell\geq 0} |a_{j,\ell}|^2\beta\left(\frac{j}{p}+1,\frac{\ell}{p}+1\right)<\infty\right\},
			\label{eq_hardyLL}
\eea
where $\beta(x,y)=\int_0^1 s^{x-1}(1-s)^{y-1}ds$ is the Euler beta function. In particular, $\mathfrak{H}^2(\E_1,\sigma)=\mathfrak{H}^2(\E_1,\om_1)$, and $\mathfrak{H}^2(\E_p,\sigma)\subsetneq \mathfrak{H}^2(\E_p,\om_p)$ when $p>1$. 
\end{Prop}
\begin{proof} We first prove \eqref{eq_hardyLL}. In view of \eqref{eq_L2eggs}, any $f\in L^2(b\E_p,\om_p)$ may be written as
		\be\label{eq_L2func}
			f(s,\theta_1,\theta_2)
		=\sum_{(j,\ell)\in\Z^2}\hat f_{j,\ell}(s)\, e^{i(j\theta_1+\ell\theta_2)},
		\ee 
where $\{\hat f_{j,\ell}(s)\}_{(j,\ell)\in\Z^2}$ are the Fourier coefficients of $f(s,.,.)$, and $\sum_{(j,\ell)\in\Z^2}||\hat f_{j,\ell}||^2_{L^2(0,1)}<~\infty$. Now, for $F\in\A(\E_p,\om_p)$, we may write
	\bes
		F(z_1,z_2)=
				\begin{cases}
			\sum_{j,\ell\geq 0}a_{j,\ell}\, z_1^jz_2^\ell,&\ \text{if}\ (z_1,z_2)\in \E_p,\\
			\sum_{j,\ell\in\Z}\widehat F_{j,\ell}(s)\, e^{i(j\theta_1+\ell\theta_2)},&\ \text{if}\
				(z_1,z_2)=\left(s^\frac{1}{2p}e^{i\theta_1},(1-s)^\frac{1}{2p}e^{i\theta_2}\right)\in b\E_p,
			\end{cases} 
	\ees
where in $\E_p$, the power series converges uniformly on compact subsets, and on $b\E_p$, the series converges in $L^2(\om_p)$. Next, for $s\in(0,1)$, $F$ is continuous on the closed polydisk $\left\{(z_1,z_2)\in\mathbb{C}^2:|z_1|\leq s^{1/2p},\,|z_2|\leq (1-s)^{1/2p} \right\}$. Hence,  
	\beas	
		\widehat F_{j,\ell}(s)
		&=&\frac{s^{\frac{j}{2p}}(1-s)^{\frac{\ell}{2p}}}{(2\pi i)^2}
				\lim_{r\rightarrow 1^-}\iint\limits_{\substack{|w_1|^{2p}=r(1-s)
		\\|w_2|^{2p}=rs }}\frac{F(w_1,w_2)}{w_1^{j+1}w_2^{\ell+1}} dw_1 dw_2
				=\begin{cases}
				a_{j,\ell}s^{\frac{j}{2p}}(1-s)^{\frac{\ell}{2p}},\ j,\ell\geq 0,\\
				0,\qquad \qquad  \text{otherwise.}
				\end{cases}
	\eeas
Moreover, 
	\bes
	\sum_{j,\ell\in\Z}||\widehat F_{j,\ell}||^2_{L^2(0,1)}=\sum_{j,\ell\geq 0}\int_{0}^1|a_{j,\ell}|^2s^{\frac{j}{p}}(1-s)^{\frac{\ell}{p}}ds
		=\sum_{j,\ell\geq 0}|a_{j,\ell}|^2\beta\left(\frac{j}{p}+1,\frac{\ell}{p}+1\right)
			<~\infty.
	\ees
Thus, we obtain the characterization in \eqref{eq_hardyLL} for a dense subspace. By taking $L^2(\om_p)$-limits of sequences in $\A(\E_p,\om_p)$, the expansion for any $f\in\h^2(\E_p,\om_p)$ can be established. The argument for \eqref{eq_hardysurf} runs along similar lines.
 
Now, since $\beta\left(\frac{j}{p}+1,\frac{\ell}{p}+1\right)\leq \beta\left(\frac{j+1}{p},\frac{\ell+1}{p}\right)$ for all $j,\ell\geq 0$, we have that $\mathfrak{H}^2(\E_p,\sigma)\subseteq \mathfrak{H}^2(\E_p,\om_p)$, with equality when $p=1$. To show strict containment for any $p>1$, we consider the series $f(s,\theta_1,\theta_2)=\sum_{j,\ell\geq 0}a_{j,\ell}s^{\frac{j}{2p}}(1-s)^{\frac{\ell}{2p}}e^{i(j\theta_1+\ell\theta_2)}$, with 
	\bes
		a_{j,\ell}=\begin{cases}
		\frac{\beta(m+1, n+1)^{-1/2}}{m n},&\ \text{when}\ \frac{j}{p}=m\in\N,
		 \frac{\ell}{p}=n\in\N,\\ 0,&\ \text{otherwise}. 
			\end{cases}
	\ees
Then $||f||_{\om_p}^2=4\pi^2\sum_{j,\ell\geq 0}|a_{j,\ell}|^2\beta\left(\frac{j}{p}+1,\frac{\ell}{p}+1\right)
		=4\pi^2\sum_{m,n\geq 0}(mn)^{-2}<\infty$, but since
		\bes
		||f||_{\sigma}^2\approx
		\sum\limits_{j,\ell\geq 0}
			|a_{j,\ell}|^2\beta\left(\frac{j+1}{p},\frac{\ell+1}{p}\right)
			\geq c\sum\limits_{m,n\geq 1}
			\frac{1}{m^2n^2}\frac{m^{1-\frac{1}{p}}n^{1-\frac{1}{p}}}
				{(m+n)^{\frac{2}{p}-2}},
		\ees
$f$ does not converge in $L^2(\sigma)$.
\end{proof}

\noindent {\em Proof of Theorem~\ref{thm_eggs}.} Fix a $p\in\N$. First, we consider $\sigma\approx {s^{\frac{1}{p}-1}(1-s)^{\frac{1}{p}-1}}\,ds\,d\theta_1\,d\theta_2
			$. It is clear that
	\beas
		z_2^{-k}\big|_{b\E_p}=(1-s)^{-{k}/{2p}}e^{-ik\theta_2}\in L^2(b\E_p,\sigma)
		 \iff k<1. 
	\eeas
Now suppose there is a $g\in \mathfrak H_1^2(\E_p^*,\sigma)\setminus\mathfrak{H}_0^2(\E_p^*,\sigma)$. Then 
	\begin{itemize}
		\item [$(i)$] $g\in L^2(\E_p,\sigma)\setminus \mathfrak{H}_0^2	(\E^*_p,\sigma)$;
		\item [$(ii)$] $(z_2g)|_{b\E_p}=\sum_{j,\ell\geq 0}
				{a_{j,\ell}s^{\frac{j}{2p}}(1-s)^{\frac{\ell}{2p}}e^{i(j\theta+\ell\phi)}}$ with
 				$\sum_{j,\ell\geq 0}|a_{j,\ell}|^2\beta\left(\frac{j+1}{p},\frac{\ell+1}{p}\right)<\infty$. 
	\end{itemize} 
Writing $g=\sum_{j,\ell\in\Z}\hat g_{j,\ell}(s)\, e^{i(j\theta_1+\ell\theta_2)}$, we obtain from $(ii)$ that
	\bes
			 \hat g_{j,\ell}=\begin{cases}
					a_{j,\ell+1},&\ \text{if}\ j\geq 0,\ell\geq -1\\
				 	0,&\ \text{otherwise}. 	
		\end{cases}			
	\ees
Thus,
	\bes
		||g||_{\sigma}^2\approx
		\sum_{j\geq 0}|a_{j,0}|^2\int\limits_{0}^{1}\frac{s^{\frac{j+1}{p}-1}ds}{1-s} 
			+\sum_{j,\ell\geq 0}|a_{j,\ell+1}|^2\beta\left(\frac{j+1}{p},\frac{\ell+1}{p}\right),
	\ees
which is finite only if $a_{j,0}=0$ for all $j\geq 0$, and $\sum_{j,\ell\geq 0}|a_{j,\ell+1}|^2\beta\left(\frac{j+1}{p},\frac{\ell+1}{p}\right)<\infty$. In that case, $g\in\mathfrak{H}_0^2(\E_p^*,\sigma)$, which contradicts $(i)$. Thus, $\mathfrak{H}_1^2(\E_p^*,\sigma)=\mathfrak{H}_0^2(\E_p^*,\sigma)$. A similar argument shows that $\mathfrak{H}_k^2(\E_p^*,\sigma)=\mathfrak{H}^2_0(\E_p^*,\sigma)$ for all $k\in\N_0$. 

In the case of $\om_p=ds\,d\theta_1\,d\theta_2$, we have that
\bes
		z_2^{-k}|_{b\E_p}=(1-s)^{-{k}/{2p}}e^{-ik\theta_2}\in L^2(b\E_p,\om_p)\iff k<p.
\ees
Thus, $z_2^{-k}	\in \mathfrak{H}^2_{k}(\E_p^*,\om_p)\setminus \mathfrak{H}^2_{k-1}(\E_p^*,\om_p)$ as long as $k\leq p-1$. For $k\geq p$, we may argue, as in the case of $\sigma$ above, that $\mathfrak{H}^2_k(\E_p^*,\om_p)=\mathfrak{H}^2_0(\E_p^*,\om_p)$.

Finally, we show that $\mathfrak{H}^2_{k-1}(\E_p^*,\omega_p)\supsetneq z_2\mathfrak{H}^2_{k}(\E_p^*,\omega_p)$ for any $k\in\N_0$. In view of the stabilization, when $k\geq p$, it suffices to show that $\mathfrak{H}^2_{p-1}(\E_p^*,\om_p)\supsetneq z_2\mathfrak{H}_{p-1}^2(\E_p^*,\om_p)$. This is clear since $z_2^{-(p-1)}\in\h^2_{p-1}(\E_p^*,\om_p)$,  but ${z_2^{-p}}\notin L^2(b\E_p,\om_p)$. For $k<p$, let 
	\beas
	f(s,\theta_1,\theta_2)&=&\sum\limits_{m\geq 0}
		(m+1)^{-\frac{k}{2p}}\left(s^{\frac{1}{2p}}e^{i\theta_1}\right)^{mp},\ \text{and}\\
	h(s,\theta_1,\theta_2)
		&=&
\left((1-s)^{\frac{1}{2p}}e^{i\theta_2}\right)^{-(k-1)}f(s,\theta_1,\theta_2).
	\eeas
Since, for any fixed $r>0$, $\beta\left(m+1,r\right)\sim (m+1)^{-r}$ as $m\rightarrow\infty$, we have that 
\bes
	||f||^2_{\om_p}=\sum\limits_{m\geq 0}(m+1)^{-\frac{k}{p}}\beta\left(m+1,1\right)\lesssim
\sum\limits_{m\geq 0}m^{-1-\frac{k}{p}}<\infty.
\ees
Thus, $z_2^{k-1}h=f\in \h^2_0(\E_p^*,\om_p)$. Moreover,  
\bes
||h||^2_{\om_p}=\sum\limits_{m\geq 0} (m+1)^{-\frac{k}{p}}\beta\left(m+1,1+\frac{1}{p}-\frac{k}{p}\right)\lesssim
\sum\limits_{m\geq 0}m^{-1-\frac{1}{p}}<\infty.
\ees
Thus, $h\in \h_{k-1}^2(\E_p^*,\om_p)$. But $|| z_2^{-1}h||^2_{\om_p}\gtrsim \sum_{m\geq 0}m^{-1} $ is not finite. Thus, there is no $g\in\mathfrak{H}_k^2(\E_p^*,\om_p)$ such that $z_2g=h$. That is, $h\in \h_{k-1}^2(\E_p^*,\om_p)\setminus z_2 \h_{k}^2(\E_p^*,\om_p)$. 
\qed

\begin{remark} The egg domains $\E_p$ may be endowed with other natural boundary measures. For example, in \cite[Def.~43]{BaLa09}, the authors consider the family of measures $\big\{\nu_\tau= f|\mathscr L|^{1-\tau}\sigma\big\}_{\tau\in[0,1]}$ on $b\E_p$, where $f$ is any positive continuous function on $b\E_p$, and
\bes
|\mathscr L|=-4|\nabla\rho|^{-3}\det\begin{pmatrix}
						0 & \rho_{\overline{z}_k} \\ 
							\rho_{z_j} & \rho_{z_j\overline{z}_k}
					\end{pmatrix}_{1\leq j,k\leq 2}
\ees
for any defining function $\rho$ of $\E_p$. The measures $\sigma$ and $\om_p$ correspond to $\nu_1$ ($f\equiv 1$) and $\nu_0$ ($f=|\nabla \rho_p|^2/4\pi^2$), respectively. It is also worth noting that the Fefferman hypersurface measure on $b\E_p$ is precisely $\nu_{2/3}$ ($f\equiv 1$). Analogous computations can be carried out to obtain explicit descriptions of the spaces $\h^2_k(\E_p^*,\nu_\tau)$. Note, in particular, that the filtration corresponding to the measure $\nu_\tau$ stabilizes at $k=\lceil p(1-\tau)+\tau\rceil -1$, where $\lceil\cdot\rceil$ is the ceiling function. 
\end{remark}

\section{Hartogs triangles: an application}\label{S:Hartogs}
 We construct filtered modules of Hardy spaces for certain power-generalized Hartogs triangles. This family of domains was first introduced in \cite{Ed16, EdMc16}. Specifically, we consider domains of the form 
	\bes
\Ha_{m/n}:=\{(z_1,z_2)\in \CC: |z_1|^m<|z_2|^n<1 \},\quad m,n \in \N,\ \gcd(m,n)=1, 
	\ees
where $\T=b\D\times b\D$ is endowed with the product measure $\sigma_\T:=\sigma_{S^1}\times\sigma_{S^1}$. Although $\Ha_{m/n}$ is not a hypersurface-deleted domain, it is a proper holomorphic image of the hypersurface-deleted domain $\D\times\D^*$ via
\bes
\Theta_{m/n}:(z_1,z_2)\mapsto (z_1^nz_2^n,z_2^m).
\ees
Note that $\Theta_{m/n}$ maps $\T$ to $\T$, and $\Theta_{m/n}^*:f\mapsto f\circ\Theta_{m/n}$ induces an isometric isomorphism from $L^2\left(\T,\sigma_\T\right)$ onto a closed subspace of $L^2(\T,\sigma_\T)$. Thus, we can deduce the Szeg{\H o} kernels for $\Ha_{m/n}$ from those for $\D\times\D^*$. To do this, we first treat the case of $\D\times\D^*$ in Subsection~\ref{SS:DXD*}. In Subsections~\ref{SS:stdHartogs} and \ref{SS:pgHartogs}, we treat the case of the standard and the power-generalized Hartogs triangles, respectively.  As done in Section~\ref{sec_examples}, we omit the measure $\sigma_{\T}$ from the notation for the relevant functions spaces. Moreover, we use polar coordinates $(\theta_1,\theta_2)$ on $\T$.

 
\subsection{Hardy spaces on $\mathbb{D}\times\mathbb{D}^*$}\label{SS:DXD*} We construct the Hardy spaces for $\D\times\D^*$ by executing the inheritance scheme in Subsection~\ref{SS:inheritance} for the triple $(\D^2,\sigma_\T,\{z_2=~0\})$. To implement the scheme, consider, for $k\in\N_0$, the following subset of $\mathcal{O}(\mathbb{D}\times\mathbb{D}^*)\cap\cont\big((\D\times\D^*)\cup\T\big)$
\beas 
\quad 
\mathcal{A}_k(\mathbb{D}\times\mathbb{D}^*)=
	\left\{F:(\D\times\D^*)\cup\T\rightarrow\C:
		F(z_1,z_2)=\left(z_2^{-k}G(z_1,z_2)\right)|_{(\D\times\D^*)\cup\T}
			\qquad \right.\\ 
				\left.\text{ for some}\ G\in\A(\D^2)=\hol(\mathbb{D}^2)\cap\cont(\D^2_\T)
					\right\}.
\eeas
For each $k\in\N_0$, set $\h^2_k(\mathbb{D}\times\mathbb{D}^*)$ to be the closure of $\mathcal{A}_k(\mathbb{D}\times\mathbb{D}^*)|_{\T}$ in $L^2(\mathbb{T})$. 

As in the case of $\D$ and $\D^*$, a precise description of these spaces in terms of Fourier series expansions can be given as follows. For $k\in\N_0$, 
	\be\label{E:HardyDXD*}
		\h^2_k(\D\times\D^*)=\left\{\sum_{(j,\ell)\in\Z^2}\!\hat f_{j,\ell}\,e^{i(j\theta_1+\ell\theta_2)}
		\in L^2(\T):\hat f_{j,\ell}=0,\ \text{if}\ \max\{j,\ell+k\}<0 \right\}.
	\ee
Moreover, the Szeg{\H o} kernel $s_k$ for $\h^2_k(\D\times\D^*)$ can be obtained by applying the Cauchy integral formula for $\D^2$ to $z_2^{k}F(z_1,z_2)$ for $F\in\A_k(\D\times\D^*)$. This yields
   \be\label{E:SzegoDXD*}   s_k(z,w)=\frac{1}{(2\pi)^2}\frac{1}{(z_2\overline{w}_2)^k(1-z_2\overline{w}_2)(1-z_1\overline{w}_1)},
	\quad z\in\D\times\D^*, w\in\T. \ee

We briefly note that in order to verify that $\h^2_k(\D\times\D^*)$ indeed satisfies the minimum criterion for being a Hardy space, we may take $\mathfrak X$ to be 
	\bes
	\mathcal{H}_k^2(\mathbb{D}\times\mathbb{D}^*)
		:=\left\{F\in\mathcal{O}(\mathbb{D}\times\mathbb{D}^*):
	||F||_{\cH^2_k(\D\times\D^*)}<\infty\right\},
	\ees
where 	
	\bes
		||F||_{\cH^2_k(\D\times\D^*)}:=
		\sup_{0<s,r<1}\left(\frac{r^{2k}}{4\pi^2}\int\limits_{0}^{2\pi}\int\limits_{0}^{2\pi}
	|F(se^{i\theta_1},re^{i\theta_2})|^2\,d\theta_1d\theta_2\right)^{\frac{1}{2}},
	\ees
with norm, a constant multiple of, $||.||_{\cH^2_k(\D\times\D^*)}$. 

\subsection{The standard Hartogs triangle}\label{SS:stdHartogs}  For the sake of exposition, we first consider,
	\bes		
		\Ha=\Ha_{1/1}=\{(z_1,z_2)\in\CC:|z_1|<|z_2|<1\},
	\ees
for which $\Theta=\Theta_{1/1}$ is, in fact, a biholomorphism. This map allows us to describe both a boundary-based construction and an exhaustion-based construction of Hardy spaces for $\Ha$. We are primarily interested in the former approach. 

 For $k\in\N_0$, let 
	\bes
		 \A_k(\Ha)=\left\{F\in\hol(\Ha)\cap\cont(\Ha\cup\T):z_2^kF(z_1,z_2)\ \text{is bounded at}\ (0,0)\right\}, 
	\ees
and $\mathfrak{H}^2_k(\Ha)$ be the closure of $\A_k(\Ha)|_{\T}$ in $L^2(\T)$. As in the case of $\D^*$ and $\D\times \D^*$, we can describe these spaces and their Szeg{\H o} kernels explicitly. 

\begin{theorem}\label{T:Hartogs} Let $k\in\N_0$. Then
	\be\label{E:isomH2}
		\mathfrak{H}^2_k(\Ha)=\left\{\sum_{(j,\ell)\in \Z^2}\!\hat f_{j,\ell}\,e^{i(j\theta_1+\ell\theta_2)}\in L^2(\T):\hat f_{j,\ell}=0,\ \text{if}\ \max\{j,\ell+j+k\}<0\right\}.
	\ee
In particular, the filtration $\{\mathfrak{H}^2_k(\Ha)\}_{k\in\N_0}$
 does not stabilize. Moreover, 
	\be\label{eq_szegoHa}
		s_k(z,w)=\frac{1}{4\pi^2}
			\frac{(z_2\overline w_2)^{-(k-1)}}{(z_2\overline w_2 -z_1\overline w_1)(1-z_2\overline w_2)},
				\quad z\in\Ha,\ w\in\T,
	\ee
is the Szeg{\H o} kernel for $\mathfrak{H}^2_k(\Ha)$.
\end{theorem} 
\begin{proof} Fix a $k\in\N_0$. Our proof relies on the fact that $\Theta^*:F|_\T\mapsto (F\circ\Theta)|_\T$ is an isometric isomorphism between  $\A_k(\Ha)|_\T$ and $\A_k(\D\times\D^*)|_\T$ in the $L^2(\T)$-norm. The isometry follows from an integration by substitution argument. For the isomorphism, note that $F\in\A_k(\Ha)$ if and only if the function 
$(z_1,z_2)\mapsto z_2^kF(z_1z_2,z_2)$ is holomorphic on $\D\times\D^*$, bounded on a neighborhood of $\{z_2=0\}$, and continuous up to $\T$. This is true if and only if
$z_2^kF(z_1z_2,z_2)=G(z_1,z_2)|_{\D\times\D^*}$ for some $G\in\A(\D^2)$. In other words, $F|_\T\in\A_k(\Ha)_\T$ if and only if $(\Theta^*F)|_\T\in \mathcal{A}_k(\D\times\D^*)|_\T$. Now, $\Theta^*$ extends to an isometry between $\h^2_k(\Ha)$ and $\h^2_k(\D\times\D^*)$ which, in terms of Fourier expansions, is given by
	\bes
		\Theta^*:\sum_{(j,\ell)\in\Z^2}\hat f_{j,\ell}\,e^{i(j\theta_1+\ell\theta_2)}\mapsto\sum_{(j,\ell)\in\Z^2}\hat f_{j,\ell}\,
e^{i\left(j\theta_1+(j+\ell)\theta_2\right)} .
	\ees
The characterization in \eqref{E:isomH2} now follows from that of $\h^2_k(\D\times\D^*)$ in \eqref{E:HardyDXD*}.

Finally, for any $F\in{\mathcal A_k(\Ha)}$, the reproducing property of the Szeg{\H o} kernel $s^{\D\times\D^*}_k $ for
$\mathfrak{H}^2_k(\D\times\D^*)$ applies to 
$\Theta^*F\in\mathcal A_k(\D\times\D^*)$. We obtain that
	\bes
		F(z)=\int_\T ({\Theta^*}F)(w)\,
		s^{\D\times\D^*}_k\left(\Theta^{-1}(z),\Theta^{-1}(w)\right)\,d \sigma_\T(w),
	\quad z\in\Ha, w\in\T.
	\ees
Now, a straightforward computation yields the reproducing property of $s_k$ as defined in \eqref{eq_szegoHa}. It is also clear that $\overline{s_k(z,\cdot)}\in\A_k(\Ha)$ for any $z\in\Ha$. 
\end{proof}

We briefly discuss an exhaustion-based construction of Hardy spaces $\cH_k^2(\Ha)$, $k\in~\N_0$, for $\Ha$, which in the case of $k=1$ is the space constructed by Monguzzi in \cite{Mo19}. For $k\in\N_0$, let 
	\bes
	\cH^2_k(\Ha)= \left\{F\in\hol(\Ha):||F||_{\cH^2_k(\Ha)}<\infty\right\},
	\ees
where
	\bes
		||F||_{\cH^2_k(\Ha)}	:=\sup\limits_{0<s,r<1}\left(\frac{r^{2k}}{4\pi^2}\int\limits_{0}^{2\pi}\int\limits_{0}^{2\pi}
	\left|F(rse^{i\theta_1},re^{i\theta_2}) \right|^2\,d\theta_1\,d\theta_2\right)^{\frac{1}{2}}.
	\ees
Rather than establish a direct isometric isomorphism, up to a factor, between $\cH^2_k(\Ha)$ and $\h^2_k(\Ha)$, we argue that $\Theta^*:F\mapsto F\circ\Theta$ is an isometric isomorphism between $\cH^2_k(\Ha)$ and $\cH^2_k(\D\times\D^*)$. From the proof of Theorem~\ref{T:Hartogs}, we know that $\Theta^*$ is an isomorphism between $\A_k(\Ha)$ and $\A_k(\D\times\D^*)$. Since these spaces are dense in the respective $\cH^2$-spaces, it suffices to show that $\Theta^*$ is an isometry from $(\A_k(\Ha),||.||_{\cH^2_k(\Ha)})$ to $(\A_k(\D\times\D^*),||.||_{\cH^2_k(\D\times\D^*)})$. This is a standard computation, by way of integration by substitution.

\subsection{The (Rational) Power-Generalized Hartogs Triangles}\label{SS:pgHartogs} 
We now consider the general case of $\Ha_{m/n}$. For $k\in\N_0$, define
	\bes
		\A_k(\Ha_{m/n})=\left\{F\in\hol(\Ha_{m/n})\cap\cont(\Ha_{m/n}\cup\T):
			z_2^kF(z_1,z_2)\ \text{is bounded at}\ (0,0)\right\}.
	\ees
Let $\h^2_k(\Ha_{m/n})$ be the closure of $\A_k(\Ha_{{m/n}})|_{\T}$ in $L^2(\T)$. As in the case $m=n=1$, using $\Theta_{m/n}^*$, we see that 
	\bes
		\h^2_k(\Ha_{m/n})=\left\{\sum_{(j,\ell)\in\Z^2}\hat f_{j,\ell}\,e^{i(j\theta_1+\ell\theta_2)}
		\in L^2(\T):\hat f_{j,\ell}=0,\ \text{if}\ \max\{j,nj+ml+mk\}<0\right\}.
	\ees
Next, we use the map $\Theta_{m/n}$ to compute the Szeg{\H o} kernel for $\h^2_k(\Ha_{m/n})$.
\begin{theorem}\label{T:pgHartogs} Let $m,n\in\N$ with $\gcd(m,n)=1$. Set 
	\bes
		\mathcal{P}_{m,n}(a,b)
	=\sum_{r=0}^{m-1}\left(a\right)^r(b)^{n-\lfloor \frac{nr}{m}\rfloor},\qquad (a,b)\in\CC.
	\ees
Then, for $k\in\N_0$,
	\be\label{eq_SzegoPowGen}
		s_k(z,w)=\frac{1}{4\pi^2}
		{\frac{(z_2\overline w_2)^{-k}\,\mathcal{P}_{m,n}
			\left(z_1\overline w_1, z_2\overline w_2\right)}
			{\left((z_2\overline w_2)^n-(z_1\overline w_1)^m\right)(1-z_2\overline w_2)}},
				\quad z\in\Ha_{m/n},\ w\in\T,
	\ee
is the Szeg{\H o} kernel for 
$\h^2_k(\Ha_{m/n})$. 
\end{theorem}
In order to prove Theorem~\ref{T:pgHartogs}, we need the following two lemmas. The proofs are straightforward applications of integration by substitution and partial fraction decompositions, so they are omitted. 

\begin{lemma}\label{lem_roots} Suppose $f\in\cont(b\D)$. Then 
		\be\label{eq_powersubs}
		 \int_{|\zeta|=1} \zeta^{n-1}f(\zeta^n)\,d\zeta 
			=\int_{|z|=1}f(z)\,dz.
		\ee
More generally, if $n\in\N$, $a\in\C\setminus S^1$, and $a_1,...,a_n$ denote the $n^{\text{th}}$-roots of $a$ (counting multiplicity). Then  
	\bea\label{eq_partfrac}
	\sum_{j=1}^n\left(\int_{|\zt|=1}\frac{f(\zeta^n)}{\zeta-a_j}\,d\zt\right)=
			n\int_{|w|=1}\dfrac{f(w)}{w-a}\,dw.
	\eea
\end{lemma}

\begin{lemma}\label{lem_partfrac2} Let $b\in\C\setminus\{0\}$ and $b_1,...,b_m$ denote the $m^{\text{th}}$-roots of $b$ (counting multiplicty). Then
	\bes
		\sum_{\ell=1}^m\frac{b_\ell^n}{(x-b_\ell^n)(y-b_\ell)}=
		\frac{m b^{n+1}\sum\limits_{p,q=0}^{m-1} c_{p,q}\,x^{p} y^{q}}{(x^m-b^n)(y^m-b)},
	\ees 
where
	\be\label{eq_coeff}
c_{p,q}=\begin{cases}
			0,& \text{if\quad}\ np+1+q\nequiv 0 \text{ (mod m)},\\
			b^{-\frac{np+1+q}{m}},& \text{if\quad}\ np+1+q\equiv 0 \text{ (mod m)}.
			\end{cases}
	\ee
\end{lemma}

\noindent {\em Proof of Theorem~\ref{T:pgHartogs}.} For any $k\in\N_0$,  $s_k$, as defined in \eqref{eq_SzegoPowGen}, satisfies $\overline{s_k(z,\cdot)}\in\A_k(\Ha_{m/n})$ for all $z\in\Ha_{m/n}$. Thus, 
it suffices to show that $s_k$ has the reproducing property for $\h^2_k(\Ha_{m/n})$. Since $|z_2|\big|_\T\equiv 1$, by Proposition~\ref{P:kernel}, we only need to prove this for $k=0$.  

Recall that $\Theta_{m/n}(\zt_1,\zeta_2)=(\zt_1^n\zeta_2^n,\zeta_2^m)$ maps $\D\times\D^*$ onto $\Ha_{m/n}$. Given $F\in \A_0(\Ha_{m/n})$ and $z=(z_1,z_2)\in\Ha_{m/n}$, let $z_{11},...,z_{1n}$ and $z_{21},...,z_{2m}$ denote the $n^{\text{th}}$-roots and $m^{\text{th}}$-roots of $z_1$ and $z_2$, respectively,  so that $F(z_1,z_2) = F(z_{1j}^n\,,z_{2\ell}^m)$ for any $1\leq j\leq n$ and $1\leq \ell\leq m$. Thus, 
		\bes
		F(z_1,z_2)=\frac{1}{mn}
		\sum_{\ell=1}^m\sum_{j=1}^n(F\circ\Theta_{m/n})\left(\frac{z_{1j}}{z_{2\ell}},z_{2\ell}\right).
		\ees 
We apply the Cauchy integral formula for $\D^2$ to $(F\circ\Theta_{m/n})\in \A_0(\D\times\D^*)=\A(\D^2)$, and obtain the following sequence of arguments.   
	\beas
		(2\pi i)^2mn F(z_1,z_2)&=&\sum_{\ell=1}^m\sum_{j=1}^n\iint_\T\frac{(F\circ \Theta_{m/n})(\zeta_1,\zeta_2)}{\left(\zeta_1-\frac{z_{1j}}{z_{2\ell}}\right)(\zt_2-z_{2\ell})}
			\,d\zt_1\,d\zt_2\\
		&=&\sum_{\ell=1}^m\int_{|\zeta_2|=1}\left(\sum_{j=1}^n\int_{|\zeta_1|=1}
		\frac{F(\zeta_1^n\zeta_2^n,\zeta_2^m)}{\zeta_1-\frac{z_{1j}}{z_{2\ell}}}
			\,d\zeta_1\right)\frac{d\zeta_2}{(\zeta_2-z_{2\ell})}\\	
		&\overset{\zt_1\zeta_2\mapsto \xi}{=}&
		\sum_{\ell=1}^m\int_{|\zeta_2|=1}\left(\sum_{j=1}^n\int_{|\xi|=1}
		\frac{F(\xi^n,\zeta_2^m)}{\left(\xi-\zeta_2\frac{z_{1j}}{z_{2\ell}}\right)}
			\,d\xi\right) \frac{d\zeta_2}{(\zeta_2-z_{2\ell})}\\	
		&\overset{\eqref{eq_partfrac}}{=}&
		n\sum_{\ell=1}^m\int_{|\zeta_2|=1}\left(\int_{|w_1|=1}
		\frac{F(w_1,\zeta_2^m)}{\left(w_1-\zeta_2^n\frac{ z_1}{z_{2\ell}^n}\right)}\,dw_1\right)
\frac{d\zeta_2}{(\zeta_2-z_{2\ell})}\\
		&=&		n\iint_{\T}\frac{F(w_1,\zeta_2^m)}{-w_1}\left(\sum_{\ell=1}^m\frac{z_{2\ell}^n}
{(\zeta_2^n\frac{ z_1}{w_1}-z_{2\ell}^n)(\zeta_2-z_{2\ell})}\right)\,dw_1\,d\zeta_2.
	\eeas
Now, by Lemma~\ref{lem_partfrac2} (with $x=\frac{\zeta_2^n z_1}{w_1}$, $y=\zeta_2$ and $b=z_2$), we have that
	\beas
	 	(2\pi i)^2mnF(z_1,z_2)
		&=&
		mn\iint_{\T}
	\frac{F(w_1,\zeta_2^m)}{-w_1}\left(\frac{z_2^{n+1}\sum\limits_{p,q=0}^{m-1}
	c_{p,q} \left(\frac{\zeta_2^n z_1}{w_1}\right)^{p}\zeta_2^{q}}{{\zeta_2^{mn}(\frac{ z_1^m}{w_1^m}-z_2^n)(\zeta_2^m-z_2)}}\right)\,dw_1\,d\zeta_2\\
				&=&
		mn\iint_{\T}
		F(w_1,\zeta_2^m)\left(\frac{\zeta_2^{m-1}
			\sum\limits_{p,q=0}^{m-1}\left( z_2^{\frac{np+1+q}{m}} c_{p,q}\right) 
		(z_1\overline w_1)^{p}(z_2\overline\zeta_2^m)^{(n+1)-\frac{np+q+1}{m}}}
		{{\left((z_2\overline\zeta_2^{m})^n-z_1^m\overline w_1^m\right)(\zeta_2^m-z_2)}}\right)\,\frac{dw_1}{w_1}\,d\zeta_2,
	\eeas
where $c_{p,q}$ are as in \eqref{eq_coeff} (with $b=z_2$). Applying \eqref{eq_powersubs} in the $\zeta_2$ variable, we get
	\beas
	 	(2\pi i)^2 F(z_1,z_2)
		&=&\iint_{\T}
		F(w_1,w_2)\frac{
			\sum\limits_{p,q=0}^{m-1}\wt c_{p,q}\,
		(z_1\overline w_1)^{p}\,(z_2\overline w_2)^{n+1-\frac{np+1+q}{m}}}
		{{(z_2^n\overline w_2^n-z_1^m\overline w_1^m)(1-z_2\overline w_2)}}\,
		\frac{dw_1}{w_1}\,\frac{dw_2}{w_2},
	\eeas
where
	\bes
	\wt c_{p,q}={z_2}^{\frac{np+1+q}{m}}c_{p,q}=\begin{cases}
			0,& \text{if}\ np+1+q\nequiv 0 \text{ (mod m)},\\
			1,& \text{if}\ np+1+q\equiv 0 \text{ (mod m)}.
			\end{cases}
	\ees
This settles our claim, once we observe that 
	\bes
		\sum_{p,q=0}^{m-1}\wt c_{p,q} 
		\left(a\right)^{p}\left(b\right)^{n+1-\frac{np+1+q}{m}}
			=\sum_{r=0}^{m-1}\left(a\right)^r(b)^{n-\lfloor \frac{nr}{m}\rfloor}
				=\mathcal{P}_{m,n}(a,b).
	\ees
\qed

In view of our minimum criterion for a Hardy space, we end this subsection with an exhaustion-based definition of Hardy spaces for $\Ha_{m/n}$. For $k\in\N_0$, let
\bes
	\cH^2_k(\Ha_{m/n})= \left\{F\in\hol(\Ha_{m/n}):||F||_{\cH^2_k(\Ha_{m/n})}<\infty\right\},
\ees
where
\bes
	||F||_{\cH^2_k(\Ha_{m/n})}:=		\sup\limits_{0<s,r<1}\left(\frac{r^{2k}}{4\pi^2}\int\limits_{0}^{2\pi}\int\limits_{0}^{2\pi}
			\left|F(r^{\frac{n}{m}}s^{\frac{1}{m}}e^{i\theta_1},re^{i\theta_2})\right|^2
				d\theta_1d\theta_2\right)^{\frac{1}{2}}.
	\ees

\subsection{$\mathbf {L^p}$-regularity of the Szeg{\H o} projection}\label{SS:Lpreg} We briefly remark on the $L^p$-mapping properties of the projection operator induced by the Szeg{\H o} kernel $s_{k} $ for $\h^2_k(\Ha_{m/n})$, $k\in\N_0$, $m,n\in\N$, $\gcd(m,n)=1$. In \cite{Mo19}, Monguzzi shows that when $k=m=n=1$, the densely defined operator 
	\bea\label{E:projection}
		\mathbb S_{k}:L^2(\T)\cap L^p(\T)&\rightarrow&\h^2_k(\Ha_{m/n}) \\
		f&\mapsto& \mathbb S_{k}f:=\left(z\mapsto\int_{\T}f(w)\cdot s_{k}(z,w)\: d\sigma_\T(w)\right)\Big|_\T\notag
	\eea
 extends to a bounded operator from $L^p(\T)$ to $L^p(\T)$ for any $p\in(1,\infty)$. This is done by realizing $\mathbb S_{1}$ as a Fourier multiplier operator on the $2$-dimensional torus $\T$. As a generalization of this fact, we note that 
$\mathbb S_{k}$ is the Fourier multiplier operator 
	\bes
		f(e^{i\theta_1},e^{i\theta_2})=\sum_{(j,\ell)\in\Z^2}\hat f_{j,\ell}\,e^{ij\theta_1}e^{i\ell\theta_2}
		\mapsto 		
	\sum_{(j,\ell)\in\Z^2}c(j,\ell)\hat f_{j,\ell}\,e^{ij\theta_1}e^{i\ell\theta_2},
		\quad  f\in L^2(\T),
	\ees
where, using the convention that $\operatorname{sgn}(0)=0$,
	\bes
		c(j,\ell)=\frac{1+\operatorname{sgn}(j+1)}{2}\cdot 
		\frac{1+\operatorname{sgn}(nj+ml+mk+1)}{2},\quad (j,\ell)\in\Z^2.
	\ees
Then, by the same argument, \eqref{E:projection} extends to a bounded operator from $L^p(\T)$ to $L^p(\T)$, $1<p<\infty$. 

The Szeg{\H o} projections considered above do not exhibit the irregularity properties of the Bergman projection, see \cite{ChZe16, EdMc16}, because the underlying Hardy spaces are supported only on the distinguished boundary of the domain. It is possible that if one considers Hardy spaces supported on the full boundary of the domain, then a stronger connection with the Bergman projection will emerge.

\bibliography{referencesFHS}{}

\begin{thebibliography}{10}

\bibitem{A50}
N.~Aronszajn.
\newblock Theory of reproducing kernels.
\newblock {\em Trans. AMS}, 68(3):337--404, 1950.

\bibitem{BaFiLe14}
L.~Baratchart, Y.~Fischer, and J.~Leblond.
\newblock Dirichlet/{N}eumann problems and {H}ardy classes for the planar
  conductivity equation.
\newblock {\em Complex Var. Elliptic Equ.}, 59(4):504--538, 2014.

\bibitem{BaLa09}
D.~E. Barrett and L.~Lanzani.
\newblock The spectrum of the {L}eray transform for convex {R}einhardt domains
  in $\mathbb{C}^2$.
\newblock {\em J. Funct. Anal.}, 257(9):2780--2819, 2009.

\bibitem{Be92}
S.~R. Bell.
\newblock {\em The {C}auchy transform, potential theory and conformal mapping}.
\newblock CRC press, 1992.

\bibitem{ChEdMc19}
D.~Chakrabarti, L.~Edholm, and J.~McNeal.
\newblock Duality and approximation of {B}ergman spaces.
\newblock {\em Adv. Math.}, 341:616--656, 2019.

\bibitem{ChSh13}
D.~Chakrabarti and M.-C. Shaw.
\newblock Sobolev regularity of the $\overline\partial$-equation on the
  {H}artogs triangle.
\newblock {\em Math. Ann.}, 356(1):241--258, 2013.

\bibitem{ChZe16}
D.~Chakrabarti and Y.~Zeytuncu.
\newblock ${L}^p$ mapping properties of the {B}ergman projection on the
  {H}artogs triangle.
\newblock {\em Proc. Amer. Math. Soc.}, 144(4):1643--1653, 2016.

\bibitem{ChCh91}
J.~Chaumat and A.-M. Chollet.
\newblock R{\'e}gularit{\'e} h{\"o}ld{\'e}rienne de l'op{\'e}rateur
  $\overline\partial$ sur le triangle de {H}artogs.
\newblock {\em Ann. Inst. Fourier (Grenoble)}, 41(4):867--882, 1991.

\bibitem{ChMc20}
L.~Chen and J.~D. McNeal.
\newblock A solution operator for $\overline\partial$ on the {H}artogs triangle
  and ${L}^p$ estimates.
\newblock {\em Math. Ann.}
\newblock (to appear).

\bibitem{Ch12}
E.~M. Chirka.
\newblock {\em Complex analytic sets}, volume~46.
\newblock Springer Science \& Business Media, 2012.

\bibitem{Da74}
H.~G. Dales.
\newblock The ring of holomorphic functions on a {S}tein compact set as a
  unique factorization domain.
\newblock {\em Proc. Amer. Math. Soc.}, 44(1):88--92, 1974.

\bibitem{Du70}
P.~L. Duren.
\newblock {\em Theory of $H^p$ Spaces}.
\newblock Academic press, 1970.

\bibitem{Ed16}
L.~Edholm.
\newblock Bergman theory of certain generalized {H}artogs triangles.
\newblock {\em Pacific J. Math.}, 284(2):327--342, 2016.

\bibitem{EdMc16}
L.~Edholm and J.~McNeal.
\newblock The {B}ergman projection on fat {H}artogs triangles: ${L}^p$
  boundedness.
\newblock {\em Proc. Amer. Math. Soc.}, 144(5):2185--2196, 2016.

\bibitem{EdMc17}
L.~Edholm and J.~McNeal.
\newblock Bergman subspaces and subkernels: Degenerate ${L}^p$ mapping and
  zeroes.
\newblock {\em J. Geom. Anal.}, 27(4):2658--2683, 2017.

\bibitem{Ha99}
T.~Hansson.
\newblock On {H}ardy spaces in complex ellipsoids.
\newblock {\em Ann. Inst. Fourier (Grenoble)}, 49(5):1477--1501, 1999.

\bibitem{HuWi20}
Z.~Huo and B.~D. Wick.
\newblock Weighted estimates for the {B}ergman projection on the {H}artogs
  triangle.
\newblock {\em J. Funct. Anal.}, 279(9):108727, 2020.

\bibitem{KeSt78}
N.~Kerzman and E.~M. Stein.
\newblock The {C}auchy kernel, the {S}zeg{\H o} kernel, and the {R}iemann
  mapping function.
\newblock {\em Math. Ann.}, 236(1):85--93, 1978.

\bibitem{La98}
L.~Lanzani.
\newblock Szeg{\H o} projection vs. potential theory for non-smooth planar
  domains.
\newblock {\em Indiana U. Math. J.}, 48(2):537--555, 1998.

\bibitem{LaSh19}
C.~Laurent-Thi{\'e}baut and M.-C. Shaw.
\newblock Solving $\overline\partial$ with prescribed support on {H}artogs
  triangles in $\mathbb{C}^2$ and $\mathbb{CP}^2$.
\newblock {\em Trans. Amer. Math. Soc.}, 371(9):6531--6546, 2019.

\bibitem{MaMi91}
L.~Ma and J.~Michel.
\newblock $\cont^{k,\alpha}$-estimates for the $\overline\partial$-equation on
  the {H}artogs triangle.
\newblock {\em Math. Ann.}, 294(1):661--675, 1992.

\bibitem{Mo19}
A.~Monguzzi.
\newblock Holomorphic function spaces on the {H}artogs triangle.
\newblock {\em Math. Nachr.}
\newblock (to appear).

\bibitem{NaPr20}
A.~Nagel and M.~Pramanik.
\newblock Bergman spaces under maps of monomial type.
\newblock {\em arXiv preprint arXiv:2002.02915}, 2020.

\bibitem{Oh02}
T.~Ohsawa.
\newblock {\em Analysis of several complex variables}.
\newblock Number 211. American Mathematical Society, 2002.

\bibitem{PoSt08}
E.~A. Poletsky and M.~I. Stessin.
\newblock Hardy and {B}ergman spaces on hyperconvex domains and their
  composition operators.
\newblock {\em Indiana Univ. Math. J.}, 57(5):2153--2201, 2008.

\bibitem{Po13}
C.~Pommerenke.
\newblock {\em Boundary behaviour of conformal maps}, volume 299.
\newblock Springer Science \& Business Media, 2013.

\bibitem{RaTo57}
H.~Rademacher and O.~Toeplitz.
\newblock {\em The enjoyment of mathematics}.
\newblock Princeton U. press, 1957.

\bibitem{Sa16}
S.~Sahin.
\newblock Poletsky-{S}tessin {H}ardy spaces on {C}omplex {E}llipsoids in
  $\mathbb{C}^n$.
\newblock {\em Complex Anal. Oper. Theory}, 10(2):295--309, 2016.

\bibitem{Sh15}
M.-C. Shaw.
\newblock The {H}artogs triangle in complex analysis.
\newblock {\em Contemp. Math.}, 646, 2015.

\bibitem{St15}
E.~M. Stein.
\newblock {\em Boundary behavior of holomorphic functions of several complex
  variables.(MN-11)}.
\newblock Princeton University Press, 2015.

\end{thebibliography}
\bibliographystyle{plain}

\end{document}